\pdfoutput=1
\RequirePackage{ifpdf}
\ifpdf 
\documentclass[pdftex]{sigma}
\else
\documentclass{sigma}
\fi

\usepackage{bbm}
\usepackage{dsfont}
\usepackage{fancyhdr}
\usepackage[all]{xy}
\usepackage{microtype}

\numberwithin{equation}{section}

\newtheorem{Theorem}{Theorem}[section]
\newtheorem*{Theorem*}{Theorem}
\newtheorem{Corollary}[Theorem]{Corollary}
\newtheorem{Lemma}[Theorem]{Lemma}
\newtheorem{Proposition}[Theorem]{Proposition}
 { \theoremstyle{definition}
\newtheorem{Definition}[Theorem]{Definition}

\newtheorem{Notation}[Theorem]{Notation}
\newtheorem{Example}[Theorem]{Example}
\newtheorem{Remark}[Theorem]{Remark} }

\newcommand{\Set}[1]{\left\lbrace #1\right\rbrace}

\newcommand{\ov}[1]{\overline{#1}}

\newcommand{\bimod}[1]{#1\text{-}\mathsf{Bimod}}
\newcommand{\lmod}[1]{#1\text{-}\mathsf{Mod}}

\newcommand{\lcomod}[1]{#1\text{-}\mathsf{Comod}}

\newcommand{\Aut}{\operatorname{Aut}}
\newcommand{\cha}{\operatorname{char}}
\newcommand{\coev}{\mathsf{coev}}
\newcommand{\coevr}{\widetilde{\mathsf{coev}}}

\newcommand{\deri}{{\rm d}}

\newcommand{\ev}{\mathsf{ev}}
\newcommand{\evr}{\widetilde{\mathsf{ev}}}
\newcommand{\End}{\mathsf{End}}
\newcommand{\Ext}{\operatorname{Ext}}
\newcommand{\Hom}{\mathsf{Hom}}
\newcommand{\wF}{\widetilde{F}}
\newcommand{\ide}{\mathsf{Id}}
\newcommand{\Img}{\operatorname{Im}}

\newcommand{\isomorph}{\stackrel{\sim}{\to}}
\newcommand{\Ob}{\mathsf{Ob}}
\newcommand{\one}{\mathds{1}}

\newcommand{\rev}{\mathrm{rev}}
\newcommand{\Tot}{\mathrm{Tot}}

\newcommand{\Alg}{\mathsf{Alg}}
\newcommand{\CoAlg}{\mathsf{CoAlg}}
\newcommand{\ComAlg}{\mathsf{ComAlg}}

\newcommand{\FrobAlg}{\mathsf{FrobAlg}}
\newcommand{\FPdim}{\mathsf{FPdim}}
\renewcommand{\dim}{\mathsf{dim}}
\newcommand{\Fun}{\mathsf{Fun}}

\newcommand{\Rep}{\mathsf{Rep}}
\newcommand{\locmod}{\mathsf{Rep}^{\mathsf{loc}}}
\newcommand{\Vect}{\mathsf{Vect}}

\newcommand{\mC}{\mathbb{C}}

\newcommand{\mZ}{\mathbb{Z}}

\newcommand{\cC}{\mathcal{C}}
\newcommand{\cD}{\mathcal{D}}

\newcommand{\cZ}{\mathcal{Z}}

\newcommand{\rI}{I}

\begin{document}
\allowdisplaybreaks

\newcommand{\arXivNumber}{2303.04493}

\renewcommand{\PaperNumber}{075}

\FirstPageHeading

\ShortArticleName{Frobenius Monoidal Functors of Dijkgraaf--Witten Categories}

\ArticleName{Frobenius Monoidal Functors of Dijkgraaf--Witten\\ Categories and Rigid Frobenius Algebras}

\Author{Samuel HANNAH~$^{\rm a}$, Robert LAUGWITZ~$^{\rm b}$ and Ana ROS CAMACHO~$^{\rm a}$}
\AuthorNameForHeading{S.~Hannah, R.~Laugwitz and A.~Ros~Camacho}

\Address{$^{\rm a)}$~School of Mathematics, Cardiff University, Abacws, Senghennydd Road,\\
\hphantom{$^{\rm a)}$}~Cardiff, CF24 4AG, Wales, UK}
\EmailD{\href{mailto:hannahs@cardiff.ac.uk}{hannahs@cardiff.ac.uk}, \href{mailto:roscamachoa@cardiff.ac.uk}{roscamachoa@cardiff.ac.uk}}

\Address{$^{\rm b)}$~School of Mathematical Sciences, University of Nottingham, University Park,\\
\hphantom{$^{\rm b)}$}~Nottingham, NG7 2RD, UK}
\EmailD{\href{mailto:robert.laugwitz@nottingham.ac.uk}{robert.laugwitz@nottingham.ac.uk}}

\ArticleDates{Received March 16, 2023, in final form September 26, 2023; Published online October 12, 2023}

\Abstract{We construct a separable Frobenius monoidal functor from \smash{$\cZ\big(\Vect_H^{\omega|_H}\big)$} to $\cZ\big(\Vect_G^\omega\big)$ for any subgroup $H$ of $G$ which preserves braiding and ribbon structure. As an application, we classify rigid Frobenius algebras in $\cZ\big(\Vect_G^\omega\big)$, recovering the classification of \'etale algebras in these categories by Davydov--Simmons [\textit{J.~Algebra} \textbf{471} (2017), 149--175, arXiv:1603.04650] and generalizing their classification to algebraically closed fields of arbitrary characteristic. Categories of local modules over such algebras are modular tensor categories by results of Kirillov--Ostrik [\textit{Adv. Math.} \textbf{171} (2002), 183--227, arXiv:math.QA/0101219] in the semisimple case and Laugwitz--Walton [\textit{Comm. Math. Phys.}, {t}o appear, arXiv:2202.08644] in the general case.}


\Classification{18M20; 18M15}

\section{Introduction}

\subsection{Motivation}

Algebraic structures play an important role in the study of conformal field theory (CFT) and topological field theory (TFT). A key structure in these applications are \emph{modular categories}, i.e., non-degenerate ribbon categories \cite{KL,Shi1}. In rational CFT, modular fusion categories appear as categories of representations over a vertex operator algebra (VOA) \cite{Hua} while modular fusion categories are utilized to construct 3d TFTs of surgery type \cite{RT91,Tur}, and appear in the classification of 3d TFTs \cite{BDSV}.

Generalizations of part of the theory and applications of modular fusion categories to low-dimensional topology have been obtained for non-semisimple (i.e., not necessarily semisimple) modular categories. These constructions include equivalent characterizations of modularity conditions \cite{Shi1}, mapping class group actions and modular functors \cite{FSS,LMSS,SW}, and partially defined non-semisimple TFTs \cite{DGGPR, KL}.
In general, it is still open whether the non-semisimple braided categories of representations of a logarithmic conformal field theories are modular~\cite{HLZ, Len}. A first example of modular categories obtained from groups are the \emph{Dijkgraaf--Witten~$($DW$)$ categories} $\cZ\big(\Vect_G^\omega\big)$ associated to a finite group $G$ and a $3$-cocycle $\omega$ on $G$ \cite{DPR}. These categories are only semisimple if the characteristic of $\Bbbk$ does not divide $|G|$ and are equivalent to representations of certain lattice VOAs.

In this paper, we focus on the study of (Frobenius) algebras in modular categories. On the one hand, modules over such algebras describe boundary conditions of the associated rational CFT associated to a certain VOA \cite{FFRS,FRS}. On the other hand, given a VOA, its possible extensions are in a one-to-one correspondence with commutative algebras in its category of representations~\cite{HKL}. This result extends to vertex operator superalgebras \cite{CKM}. These results give us motivation for classifying algebra objects in $\cZ\big(\Vect_G^\omega\big)$. Many categories of representations of a VOA are pointed fusion categories, like the case of, e.g., lattice VOAs coming from an even, integral lattice (here, $G=\Lambda^*/\Lambda$ is the discriminant form of the lattice, note that this $G$ is abelian) \cite{DL,Len}. In this sense, an important family of vertex operator algebras are the holomorphic ones, those whose category of representations is simply $\Vect$. Given a certain group $G$, one can take the so-called orbifold of a holomorphic VOA, see, e.g., \cite{DongRenXu, Moeller}. Its category of representations will be then equivalent to $\cZ\big(\Vect_G^\omega\big)$ \cite{McRae}.

Given a commutative algebra $A$ in a braided tensor category $\cC$ one defines a braided tensor category $\locmod_\cC(A)$ of \emph{local modules} \cite{KO,LW3,Par,Sch}. Such categories of local modules have been of particular interest in the mathematical physics literature, see, e.g., \cite{FFRS, FRS}.
For instance, categories of local modules relate the representations of a VOA to those of its extensions \cite{CKM,HKL, KO}. Given a \emph{rigid Frobenius algebra} (i.e., a connected commutative special Frobenius algebra) in~$\cC$, it was shown that the rigid monoidal category $\locmod_\cC(A)$ of local modules is again modular (see~\cite{KO} in the semisimple case, and \cite{LW3} in the general case). Such rigid Frobenius algebras were classified for the semisimplification of $U_q(\mathfrak{sl}_2)$-modules \cite{KO}, for the Drinfeld center of modules over a finite group \cite{Dav3,LW3}, and for DW categories $\cZ\big(\Vect_{\Bbbk G}^\omega\big)$ in $\cha\Bbbk=0$ \cite{DS}.

In the present paper, we construct Frobenius monoidal functors. Given two monoidal cate\-gories~$\cC$ and $\cD$, a Frobenius monoidal functor $F\colon \cC\to \cD$ comes with a choice of natural transformations
\begin{gather*}
\mu_{V,W}\colon\ F(V)\otimes F(W)\to F(V\otimes W),\qquad
\nu_{V,W}\colon\ F(V\otimes W)\to F(V)\otimes F(W),
\end{gather*}
which make $F$ a lax and oplax monoidal functor and satisfy compatibility conditions which are analogue to those of a product and coproduct of a Frobenius algebra. While any monoidal functor is, in particular, Frobenius monoidal, for general Frobenius monoidal functors, like those considered in this paper, $F(V)\otimes F(W)$ and $F(V\otimes W)$ are not isomorphic. However, any Frobenius monoidal functor sends Frobenius algebras in $\cC$ to Frobenius algebras in $\cD$.
Frobenius monoidal functors have recently appeared in different contexts in the quantum algebra literature, see, e.g., \cite{FHL,WINART,Yad}. Here, we construct Frobenius monoidal functors to categories of the form~$\cZ\big(\Vect_G^\omega\big)$. These functors are separable, so that $F(V\otimes W)$ is naturally a direct summand of $F(V)\otimes F(W)$, and compatible with braidings whence they preserve connected commutative Frobenius algebras. We apply these functors to classify rigid Frobenius algebras in $\cZ\big(\Vect_G^\omega\big)$ for a field of arbitrary characteristic.

Algebra objects in $\Vect_G^\omega$ were classified up to equivalence of the associated $\Vect_G^\omega$-module categories and representatives are given by twisted group algebras $A(N,\kappa)$ associated to a normal subgroup $N$ and a $2$-cocycle $\kappa$ such that $\deri \kappa= \omega |_N$ \cite{Nat, Ost}, see also \cite{WINART} for explicit Frobenius algebra structures on these algebras. In this paper we find conditions for the existence of lifts of these twisted group algebras to rigid Frobenius algebras in $\cZ\big(\Vect_G^\omega\big)$ in terms of homological algebra data building on results of \cite{DS}. To these central lifts of the twisted group algebra $A(N,\kappa)$ one can then associate tensor categories of representations whose centers are given by local modules.

\subsection{Statements of results}

Let $\Bbbk$ be an algebraically closed field of arbitrary characteristic. We fix a finite group $G$, with a~subgroup $H$, and a $3$-cocycle $\omega$ on $G$ and prove the following result.

\begin{theorem*}[see Propositions~\ref{prop:Frob} and \ref{prop:braidedFrob}]
There is a separable Frobenius monoidal functor $\rI\colon \cZ\big(\Vect_H^{\omega|_{H}}\big)\to \cZ\big(\Vect_G^\omega\big)$. This tensor functor $\rI$ is compatible with braidings and preserves ribbon twists.
\end{theorem*}

Using the Frobenius monoidal functors $\rI$, we classify rigid Frobenius algebras in $\cZ\big(\Vect_G^\omega\big)$ generalizing results by Davydov--Simmons \cite{DS} to the non-semisimple case. In fact, all rigid Frobenius algebras in $\cZ\big(\Vect_G^\omega\big)$ are of the form $A=\rI(B)$ for some subgroup $H$ of $G$, and $B=B(N,\kappa,\varepsilon)$ a rigid Frobenius algebra in \smash{$\cZ\big(\Vect_H^{\omega|_H}\big)$} with $\dim_{\Bbbk}B_1=1$. Such algebras $B$ are parametrized by certain elements $\varepsilon\oplus \kappa$ of the second cohomology group $\widetilde{H}^2_{\Tot}(H,N,\Bbbk^\times)$ of a~truncated total complex \smash{$\big(\widetilde{F}^\bullet_{\Tot}(H,N,\Bbbk^\times), \deri_{\Tot}\big)$} which computes the group cohomology of the semi-direct product $H\ltimes N$, described in Appendix~\ref{appendix:coh-crossed} and \cite[Appendix~A]{DS}.

The following result recovers, and extends to arbitrary characteristic, the classification of connected \'etale algebras in $\cZ\big(\Vect_G^\omega\big)$ in \cite[Theorem~3.15]{DS}.

\begin{theorem*}[see Theorem~\ref{thm:classification}]\label{thm2}
Every connected \'etale algebra in $\cZ\big(\Vect_G^\omega\big)$ is isomorphic to one of the form $A(H,N,\kappa, \epsilon)$ for some choice of data $H$, $N$, $\gamma$, $\epsilon$, where
\begin{itemize}\itemsep=0pt
 \item $H$ is a subgroup of $G$, with $N$ a normal subgroup of $H$,
 \item $\kappa\colon N\times N\to \Bbbk^\times$ is a function satisfying $\deri(\kappa)=\omega|_N$,
 \item $\epsilon\colon H\times N \to \Bbbk^\times$ is a function such that
 $\deri_{\Tot}(\varepsilon\oplus\kappa)=\tau\oplus \gamma \oplus \omega$,
 \item the compatibility $\epsilon(n,m)=\frac{\kappa(nmn^{-1},n)}{\kappa(n,m)}$ holds for all $n,m\in N$.
\end{itemize}
Every such connected \'etale algebra has trivial twist and is a rigid Frobenius algebra if and only if~$|N|\cdot |G:H|\neq 0 \in \Bbbk^{\times}$.
\end{theorem*}
We provide explicit formulas for the Frobenius algebras $A(H,N,\kappa,\epsilon)$ in Lemma~\ref{lem:AFrobenius}.

An interpretation of Section~\ref{thm2} is that the algebras $B(N,\kappa,\epsilon)$ are lifts of the twisted group algebras $A(N,\kappa)$ in \smash{$\Vect_H^{\omega|_H}$} to the center $\cZ\big(\Vect_H^\omega\big)$. These twisted group algebras were used to classify indecomposable module categories over $\Vect_H^\omega$ \cite{Nat,Ost} and are separable Frobenius algebras \cite{WINART}.
Our results show that lifts of these algebras to the center along the forgetful functor are parametrized by functions $\epsilon\colon H\times N\to \Bbbk^\times$ satisfying the conditions from Section~\ref{thm2}.

The category of local modules $\locmod_{\cZ(\Vect_{G}^\omega)}(A)$ over a rigid Frobenius algebra $A$ as in Section~\ref{thm2} is a modular category by \cite[Theorem 4.12]{LW3} and \cite[Theorem~4.5]{KO} in the semisimple case. In fact, \cite[Theorem~3.16]{DS} shows that such modular categories are equivalent as ribbon categories to~$\cZ\big(\Vect_{H/N}^{\ov{\omega}}\big)$ for a $3$-cocycle $\ov{\omega}$ on $H/N$ such that its pullback to $H$ via the quotient homomorphism is equivalent to $\omega|_H$.

In Section~\ref{sec:expl}, we classify all rigid Frobenius algebras in $\cZ\big(\Vect_G^\omega\big)$ for an odd dihedral group $G=D_{2m+1}$, up to \emph{isomorphism} of algebras in $\cZ\big(\Vect_G^\omega\big)$ rather than up to equivalence of their categories of local modules.

The paper is structured as follows. In Section~\ref{sec:background}, we recall the necessary background on (non-semisimple) modular categories, algebraic structures in ribbon categories, and local modules, concluding with a brief review of DW categories. Section~\ref{sec:results} contains the results of the paper, starting with a discussion on DW categories associated to quotient groups, followed by the construction of the Frobenius monoidal functors, and the classification of rigid Frobenius algebras in DW categories. In Appendix~\ref{appendix}, we include basic definitions from group cohomology and several cocycle identities used throughout the text.

\section{Background}\label{sec:background}

\subsection{Modular tensor categories}

Throughout this paper, we fix $\mathbbm{k}$ to be an algebraically closed field of arbitrary characteristic.

In this section, we collect some basic definitions, see, e.g.,~\cite{EGNO} for details. A \textit{monoidal category}~$\cC$ consists of a tuple $( \cC,\otimes, \mathbbm{1},\alpha,\lambda,\rho)$ where $\cC$ is a category, $\otimes \colon \cC \times \cC \to \cC$ is a bifunctor, $\mathbbm{1} \in \Ob (\cC)$, $\alpha_{X,Y,Z} \colon ( X \otimes Y) \otimes Z \to X \otimes ( Y \otimes Z)$ is a natural isomorphism for each $X, Y, Z \in \Ob (\cC)$, and $\lambda_X \colon \mathbbm{1} \otimes X \to X$ and $\rho_X \colon X \otimes \mathbbm{1} \to X$ are natural isomorphisms for all $X \in \Ob (\cC)$,
satisfying coherence axioms (pentagon and triangle).
A functor $F\colon \cC\to \cD$ between two monoidal categories is a \emph{monoidal functor} if there exist natural isomorphisms
\[\mu^F\colon\ F(X)\otimes^\cD F(Y)\to F\big(X\otimes^\cC Y\big), \qquad \one^\cD \to F\big(\one^\cC\big),\]
satisfying certain coherence conditions, see \cite[Definition~2.4.1]{EGNO}.

A monoidal category is called \textit{rigid} if it comes equipped with left and right dual objects. That means, for every $X \in \Ob (\cC)$ there exists respectively an object $X^* \in \Ob (\cC)$ with evaluation and coevaluation maps $\ev_X \colon X^* \otimes X \to \mathbbm{1}$ and $\coev_X \colon \mathbbm{1} \to X \otimes X^*$, as well as an object~${{}^*X \in \Ob (\cC)}$ with evaluation and coevaluation maps $\evr_X \colon X \otimes {}^*X \to \mathbbm{1}$ and $\coevr_X \colon \mathbbm{1} \to {}^*X \otimes X$ satisfying in both cases the usual conditions.
If a rigid monoidal category comes equipped with isomorphisms~$j_X \colon X \to X^{**}$ natural in $X \in \Ob(\cC)$ and satisfying that $j_{X \otimes Y}=j_X \otimes j_Y$, then it is called \textit{pivotal}. The \textit{quantum dimension} of an object $X$ in a pivotal category is the composition $\mathrm{qdim}_j (X):=\ev_{X^*} ( j_X \otimes \ide_{X^*} ) \coev_X \in \End_\cC ( \mathbbm{1} )$.

A $\mathbbm{k}$-linear abelian category $\cC$ is \textit{locally finite} if, for any two objects $V,W \in \Ob (\cC)$, $\Hom_\cC ( V,W)$ is a finite-dimensional $\mathbbm{k}$-vector space and every object has a finite filtration by simple objects. Further, we say $\cC$ is \textit{finite} if $\cC$ is equivalent to a category of finite-dimensional modules over a~finite-dimensional $\Bbbk$-algebra. A \textit{tensor category} is a locally finite, rigid, monoidal category such the tensor product is $\mathbbm{k}$-linear in each slot and the monoidal unit is a simple object of the category.

A monoidal category $\cC$ is called \textit{braided} if it comes equipped with natural isomorphisms $c_{X,Y} \colon X \otimes Y \to Y \otimes X$ for all $X, Y \in \Ob (\cC)$, called the \textit{braiding}, that are compatible with the monoidal structure of the category. This means, the braiding satisfies the so-called \textit{hexagon identities} for any three objects $X, Y, Z \in \Ob (\cC)$:
\begin{gather*}
 \xymatrix{& X \otimes ( Y \otimes Z) \ar[rr]^{c_{X, Y \otimes Z}} && ( Y \otimes Z ) X \ar[dr]^{\alpha_{Y,Z,X}} & \\
 ( X \otimes Y ) \otimes Z \ar[ur]^{\alpha_{X,Y,Z}} \ar[dr]_{c_{X,Y} \otimes \ide_Z} &&&& Y \otimes ( Z \otimes X ), \\
 & ( Y \otimes X ) \otimes Z \ar[rr]^{\alpha_{Y,X,Z}} && Y \otimes ( X \otimes Z ) \ar[ur]_{\ide_Y \otimes c_{X,Z}}
 }
 \\
 \xymatrix{ & ( X \otimes Y ) \otimes Z \ar[rr]^{c_{X \otimes Y,Z}} && Z \otimes ( X \otimes Y ) \ar[dr]^{\alpha^{-1}_{Z,X,Y}} & \\
 X \otimes ( Y \otimes Z ) \ar[ur]^{\alpha^{-1}_{X,Y,Z}} \ar[dr]_{\ide_X \otimes c_{Y,Z}} &&&& ( Z \otimes X ) \otimes Y. \\
 & X \otimes ( Z \otimes Y ) \ar[rr]^{\alpha^{-1}_{X,Z,Y}} && ( X \otimes Z ) \otimes Y \ar[ur]_{c_{X,Z} \otimes \ide_Y}
 }
\end{gather*}

An example of a braided category is that of the \emph{Drinfeld center} (or \emph{monoidal center}, or simply \emph{center}) of a monoidal category $\cC$. Its objects are pairs $\big( X, c^X \big)$ where $X \in \Ob (\cC)$ and $c^X_V \colon X \otimes V \to V \otimes X$ (for any $V \in \Ob (\cC)$) is a natural isomorphism called the \textit{half-braiding} satisfying that{\samepage
\begin{equation*}
 c^X_{V \otimes W} =\big( \ide_{V} \otimes c^X_{W} \big) \big( c^X_{V} \otimes \ide_{W} \big).
\end{equation*}
The braiding of this category is given by $c_{(X, c^X),(Y,c^Y)}:=c^X_{Y}$.}

A \emph{ribbon category} is a braided tensor category $\cC$ together with a \emph{ribbon twist}, i.e., a natural isomorphism $\theta_X\colon X\to X$ which satisfies
\begin{align}\label{eq:ribbontwist-cond}
\theta_{X \otimes Y} = (\theta_X \otimes \theta_Y) c_{Y,X} c_{X,Y},\qquad \theta_\one = \ide_\one,\qquad (\theta_X)^* = \theta_{X^*}.
\end{align}
A tensor functor $F\colon \cC\to \cD$ between ribbon categories $\cC$, $\cD$ with ribbon twists $\theta^\cC$, $\theta^\cD$ is a \emph{ribbon tensor functor} if it commutes with the ribbon structures in the sense that $F\big(\theta^{\cC}_V\big)=\theta^\cD_{F(V)}$. If $F$ is part of an equivalence of categories, then $\cC$ and $\cD$ are equivalent as ribbon categories.

In order to define modular tensor categories, we require the notion of non-degeneracy of a~braided category. We say that an object $X$ \emph{centralizes} another object $Y$ of $\cC$ if
\[c_{Y,X}c_{X,Y}=\ide_{X\otimes Y}.\]
A braided finite tensor category $\cC$ is \emph{non-degenerate} if the only objects $X$ that centralize \emph{all} objects of $\cC$ are of the form $X=\one^{\oplus n}$ \cite[Section~8.20]{EGNO}.
Equivalently, $\cC$ is non-degenerate if and only if it is \emph{factorizable}, i.e., there is an equivalence of braided monoidal categories~$\cZ(\cC)\simeq \cC^{\rev}\boxtimes\cC$, where $\cC^\rev$ is $\cC$ as a tensor category, but with reversed braiding given by the inverse braiding~\cite{Shi1}. If $\cC$ is a fusion category (i.e., a semisimple finite tensor category) then the above notion of non-degeneracy is equivalent to the commonly used condition that the $S$-matrix is non-singular.
A key definition for this paper is the concept of modular category that allows for using general finite tensor categories which are not necessarily non-semisimple.

\begin{Definition}[{\cite{KL,Shi1}}]
A braided finite tensor category is \emph{modular} if it is a non-degenerate ribbon category.
\end{Definition}

\subsection{Frobenius algebras in tensor categories}

In this section, let $\cC=( \cC,\otimes, \mathbbm{1},\alpha,\lambda,\rho)$ be a pivotal finite tensor category.

\begin{Definition} \label{def:co-alg-Frob}\quad
\begin{itemize}\itemsep=0pt
 \item[(a)] An \textit{algebra} in $\cC$ is a triple $( A,m,u)$, with $A \in \Ob (\cC)$, and $m \colon A \otimes A \to A$ (multiplication), $u \colon \mathbbm{1} \to A$ (unit) being morphisms in $\cC$,
 satisfying unitality and associativity constraints:
 \begin{align*}
 m (m \otimes \ide_A) = m(\ide_A \otimes m) \alpha_{A,A,A}, \qquad
 m (u \otimes \ide_A) = \lambda_A, \qquad m(\ide_A \otimes u) = \rho_A.
 \end{align*}

 \item[(b)] A \textit{coalgebra} in $\cC$ is a triple $( C,\Delta,\varepsilon)$, where $C \in \Ob(\cC)$, and $\Delta \colon C \to C \otimes C$ (comultiplication) and $\varepsilon \colon C \to \mathbbm{1}$ (counit) are morphisms in $\cC$,
 satisfying counitality and coassociativity constraints:
 \begin{align*}
 \alpha_{C,C,C}(\Delta \otimes \ide_C) \Delta = (\ide_C \otimes \Delta)\Delta, \qquad
 (\varepsilon \otimes \ide_C)\Delta = \lambda_C^{-1}, \qquad (\ide_C \otimes \varepsilon)\Delta = \rho_C^{-1}.
 \end{align*}

 \item[(c)] A \textit{Frobenius algebra} in $\cC$ is a tuple $ ( A, m, u, \Delta, \varepsilon )$, where
 $( A,m,u)$ is an algebra and $ ( A,\Delta,\varepsilon )$ is a coalgebra so that
 \begin{equation*}
 ( m \otimes \ide_A ) \alpha^{-1}_{A,A,A} ( \ide_A \otimes \Delta )=\Delta m = ( \ide_A \otimes m ) \alpha_{A,A,A} ( \Delta \otimes \ide_A ).
 \end{equation*}
\end{itemize}
\end{Definition}

\begin{Remark} \label{rem:Frobpairing}
 Alternatively, a Frobenius algebra in $\cC$ is a tuple $ ( A,m,u,p,q )$, where $ ( A, m, u )$ is an algebra, $p \colon A \otimes A \to \mathbbm{1}$ and $q \colon \mathbbm{1} \to A \otimes A$ are morphisms in $\cC$ satisfying an invariance condition,
 \[
 p ( \ide_A \otimes m ) \alpha_{A,A,A} = p ( m \otimes \ide_A ),
 \]
 and the `snake' equations.
To convert from $ ( A,m,u,p,q )$ to~$ ( A,m,u,\Delta,\varepsilon )$ in the previous definition, take
\[
\Delta:= ( m \otimes \ide_A )\alpha_{A,A,A}^{-1} (\ide_A \otimes q )\rho_A^{-1} \qquad \text{and} \qquad {\varepsilon:=p (u \otimes \ide_A)\rho_A^{-1}}.
\]
On the other hand, to convert from $( A,m,u,\Delta,\varepsilon )$ to $( A,m,u,p,q)$, take
$p:= \varepsilon_A m_A$ and $q:= \Delta_A u_A$, cf.\ \cite{FS}.
\end{Remark}

\begin{Definition}\quad
\begin{enumerate}\itemsep=0pt
 \item[(a)] An algebra $A$ in $\cC$ is \textit{indecomposable} if it is not isomorphic to a direct sum of non-trivial algebras in $\cC$.
 \item[(b)] An algebra $A$ in $\cC$ is \textit{connected} (or \textit{haploid}) if
 $\dim_{\mathbbm{k}}\Hom_{\mathcal C} ( \mathbbm{1}, A ) = 1$.
 \item[(c)] An algebra $A$ in $\cC$ is \textit{separable} if there exists a morphism $\Delta' \colon A \to A \otimes A$ in $\cC$ so that $m \Delta' = \ide_A$ as maps in $\cC$ with
 \begin{equation*}
( \ide_A\otimes m ) \alpha_{A, A, A} ( \Delta'\otimes \ide_A ) = \Delta' m = ( m \otimes \ide_A ) \alpha^{-1}_{A, A, A} ( \ide_A\otimes \Delta' ).
 \end{equation*}
 \item[(d)] A Frobenius algebra $ ( A,m,u,\Delta,\varepsilon )$ in $\cC$ is
 \textit{special} if $m \Delta = \beta_A\ide_A$ and $\varepsilon u = \beta_1 \; \ide_\mathbbm{1}$ for nonzero $\beta_A,\beta_\one \in \mathbbm{k}^\times$.
 \item[(e)] If $\cC$ is braided with braiding $c$, we call an algebra $A$ in $\cC$ \emph{commutative} if $m c_{A,A}=m$.
 \item[(f)] A separable commutative algebra in $\cC$ is also called an \emph{\'etale algebra}.
 \end{enumerate}
\end{Definition}

Recall that a ribbon category $\cC$ is, in particular, pivotal with pivotal structure $j$ discussed, for example, in \cite[Section~2.5]{LW3}.

\begin{Proposition}[{\cite[Proposition 3.12]{LW3}}] \label{prop:rigidFrobcharacterization}
The following statements are equivalent for a connected commutative algebra $A$ in a ribbon category $\cC$ with twist $\theta$:
\begin{enumerate}\itemsep=0pt
 \item[$(a)$] $A$ is separable with $\dim_j A \neq 0$ and $\theta_A=\ide_A$;
 \item[$(b)$] $A$ is a special Frobenius algebra;
 \item[$(c)$] $A$ admits a morphism $\varepsilon\colon A\to \one$ such $\varepsilon u=\ide_\one$, $\varepsilon m$ is non-degenerate, $\dim_j(A)\neq 0$, and~${\theta_A=\ide_A}$.
\end{enumerate}
\end{Proposition}

If $A$ satisfies the equivalent conditions from Proposition~\ref{prop:rigidFrobcharacterization}, then we say that $A$ is a \emph{rigid Frobenius algebra}. If $\cC$ is semisimple, the conditions in (c) on a connected commutative algebra in $\cC$ recover the definition of a rigid $\cC$-algebra used in \cite{KO} to show that the category of local modules are semisimple. We recall a version of this result which holds even if $\cC$ is not semisimple in Theorem~\ref{thm:locmodular}.

\subsection{Frobenius monoidal functors}

In this section, we recall the definition of a Frobenius monoidal functor and include basic results about such functors preserving algebraic structures in tensor categories. Let $\cC$ and $\cD$ be two monoidal categories.

\begin{Definition}
A \emph{lax monoidal functor} from $\cC$ to $\cD$ consists of
\begin{itemize}\itemsep=0pt
 \item a functor $F \colon \cC \to \cD$,
 \item a natural transformation $\mu_{V,W}\colon F(V)\otimes F(W)\longrightarrow F(V\otimes W)$, and
 \item a morphism $\eta\colon \one \longrightarrow F(\one)$,
\end{itemize}
for any $V,W \in \Ob (\cC)$, subject to the compatibility conditions:
\begin{gather}
\xymatrix{
( F (U) \otimes F (V) ) \otimes F (W) \ar[rr]^{\alpha_{F(U),F(V),F(W)}} \ar[d]_{\mu_{U,V}\otimes \ide_{F(W)}} && F (U) \otimes( F (V) \otimes F (W)) \ar[d]^{\ide_{F(U)} \otimes \mu_{V,W}} \\
F ( U \otimes V ) \otimes F (W) \ar[d]_{\mu_{U \otimes V,W}} && F (U) \otimes F ( V \otimes W ) \ar[d]^{\mu_{U,V \otimes W}} \\
F ( ( U \otimes V ) \otimes W ) \ar[rr]^{F( \alpha_{U,V,W} )} && F (U \otimes ( V \otimes W )),
} \nonumber\\
\begin{split}
& \xymatrix{
\one \otimes F (U) \ar[rr]^{\eta \otimes \ide_{F(U)}} \ar[d]_{\lambda_{F(U)}} && F (\one) \otimes F (U) \ar[d]^{\mu_{\one,U}} \\
F (U) && F ( \one \otimes U ), \ar[ll]_{F (\lambda_U)}
} \qquad \xymatrix{
F (U) \otimes \one \ar[rr]^{\ide_{F(U)} \otimes \eta} \ar[d]_{\rho_{F(U)}} && F (U) \ar[d]^{\mu_{U,\one}} \\
F (U) && F ( U \otimes \one). \ar[ll]_{F ( \rho_U)}
}\end{split}
\label{laxmon}
\end{gather}
We will denote the lax monoidal structure as $( \mu,\eta)$.
\end{Definition}

\begin{Definition}
An \emph{oplax monoidal functor} from $\cC$ to $\cD$ consists of
\begin{itemize}\itemsep=0pt
 \item a functor $F \colon \cC \to \cD$,
 \item a natural transformation $\nu_{V,W}\colon F(V\otimes W)\longrightarrow F(V)\otimes F(W)$ and
 \item a morphism $\epsilon\colon F(\one)\longrightarrow \one$,
\end{itemize}
for any $V,W \in \Ob (\cC)$, subject to compatibility conditions analogous to those of lax monoidal~\eqref{laxmon}, but with their arrows reversed. We will denote the oplax monoidal structure as $( \nu,\epsilon)$.
\end{Definition}

\begin{Definition}\label{def:Frobmonoidal}
A \emph{Frobenius monoidal functor} $F\colon \cC\to \cD$ between two monoidal categories $\cC$,~$\cD$ is a bilax monoidal functor, i.e., comes with a lax monoidal structure $(\mu,\eta)$, and an oplax monoidal structure $(\nu, \epsilon)$, where
\begin{align*}
 & \mu_{V,W}\colon\ F(V)\otimes F(W)\longrightarrow F(V\otimes W), \qquad \nu_{V,W}\colon\ F(V\otimes W)\longrightarrow F(V)\otimes F(W),\\
 & \eta\colon\ \one \longrightarrow F(\one), \qquad \epsilon\colon\ F(\one)\longrightarrow \one,
\end{align*}
for any objects $V$, $W$ of $\cC$,
satisfying the additional compatibility conditions
\begin{align}\label{frobmon1}
\vcenter{\hbox{\xymatrix{
F(V)\otimes F(W\otimes U)\ar[rr]^{\ide_{F(V)}\otimes \nu_{W,U}}\ar[d]_{\mu_{V,W\otimes U}} && F(V)\otimes (F(W)\otimes F(U))\ar[d]^{\alpha^{-1}_{F(V),F(W),F(U)}}\\
F(V\otimes (W\otimes U))\ar[d]_{F(\alpha^{-1}_{V,W,U})}&& (F(V) \otimes F(W))\otimes F(U)\ar[d]^{\mu_{V,W}\otimes \ide_{F(U)}}\\
F((V\otimes W)\otimes U)\ar[rr]^{\nu_{V\otimes W,U}}&&F(V\otimes W) \otimes F(U),
}}}
\\
\vcenter{\hbox{\xymatrix{
F(V\otimes W)\otimes F(U)\ar[rr]^{\nu_{V,W}\otimes \ide_{F(U)}}\ar[d]_{\mu_{V\otimes W,U}} && (F(V)\otimes F(W))\otimes F(U)\ar[d]^{\alpha_{F(V),F(W),F(U)}}\\
F((V\otimes W)\otimes U)\ar[d]_{F(\alpha_{V,W,U})}&& F(V) \otimes (F(W)\otimes F(U))\ar[d]^{\ide_{F(V)} \otimes \mu_{W,U}}\\
F(V\otimes (W\otimes U))\ar[rr]^{\nu_{V,W\otimes U}}&&F(V)\otimes F(W \otimes U).
}}}\label{frobmon2}
\end{align}
\end{Definition}

We say that a Frobenius monoidal functor is \emph{separable} if for any objects $V$, $W$ of $\cC$,
\[
 \mu_{V,W}\circ\nu_{V,W} = \ide_{F(V\otimes W)}.
\]
For details on these definitions see, e.g., \cite[Section~3.5]{AA}.

We are also interested in compatibility conditions of Frobenius monoidal functors with braidings.
Denote a braided monoidal category by $( \cC, c )$. Given two braided monoidal categories~$( \cC, c)$ and~$( \cD, d )$, a \textit{braided lax monoidal functor} is a lax monoidal functor $F \colon \cC \to \cD$ which in addition satisfies
\begin{equation*}
 \xymatrix{F (X) \otimes F ( Y ) \ar[d]_{\mu_{X,Y}} \ar[rr]^{d_{F(X),F(Y)}} && F ( Y) \otimes F (X) \ar[d]^{\mu_{Y,X}} \\ F ( X \otimes Y ) \ar[rr]^{F(c_{X,Y})} && F( Y \otimes X )}
\end{equation*} for any $X, Y \in \Ob (\cC)$.
The notion of braided oplax monoidal functor is analogous to this one. We note that (braided lax/oplax) Frobenius monoidal functors preserve algebraic structures in the respective categories.
\begin{Proposition}\label{prop:Frob-perserved}
Let $F\colon \cC \to \cD$ be a Frobenius monoidal functor.
\begin{enumerate}\itemsep=0pt
 \item[$(a)$] If $A$ is a $($co$)$algebra in $\cC$, then $F(A)$ is a $($co$)$algebra in $\cD$. In fact, $F$ restricts to a functor $F\colon \Alg(\cC)\to \Alg(\cD)$, and a functor $F\colon \CoAlg(\cC)\to \CoAlg(\cD)$.
 \item[$(b)$] If $A$ is a Frobenius algebra in $\cC$, then $F(A)$ is a Frobenius algebra in $\cD$. In fact, $F$ restricts to a functor $F\colon \FrobAlg(\cC)\to \FrobAlg(\cD)$.
 \item[$(c)$] If $F$ is, in addition, separable and $\epsilon \circ \eta\neq 0$ and $A$ a special Frobenius algebra in $\cC$, then~$F(A)$ is a special Frobenius algebra in $\cD$.
\end{enumerate}
\end{Proposition}

\begin{Definition} Take $A:=(A,m_A,u_A)$, an algebra in $\cC$.
A {\it right $A$-module in $\cC$} is a pair $(M, \rho_M)$, where $M \in \cC$, and $\rho_M:=\rho_M^A \colon M \otimes A \to M$ is a morphism in $\cC$ so that \[\rho_M(\rho_M \otimes \ide_A) = \rho_M(\ide_M \otimes m_A)\alpha_{M,A,A} \qquad \text{ and } \qquad r_M = \rho_M(\ide_M \otimes u_A).\] A {\it morphism} of right $A$-modules in $\cC$ is a morphism $f \colon M \to N$ in $\cC$ so that $f\rho_M = \rho_N (f \otimes \ide_A)$. Right $A$-modules in $\cC$ and their morphisms form a category, which we denote by~$\cC_A$.
The categories ${}_A \cC$ of {\it left $A$-modules $(M, \lambda_M:=\lambda_M^A \colon A \otimes M \to M)$ and ${}_A \cC_A$ of {\it $A$-bimodules} in $\cC$} are defined likewise.
\end{Definition}

It follows that given a Frobenius monoidal functor $F$, or, any lax monoidal functor, and an algebra $A$ in $\cC$, $F$ induces a functor $F\colon {}_A\cC\to {}_{F(A)}\cD$. Similar statements hold for left modules, and right\slash left comodules where an oplax monoidal functor is needed.

\subsection{Local modules}\label{subsec:local}

In the following, we recall local modules over commutative algebras in a braided category $\cC$ \cite{KO,LW3,Par,Sch}.

\begin{Definition}
Let $\Rep_{\cC}(A)$ denote the category whose objects are pairs $(V,a_V^r) \in \cC_A$, and morphisms are morphisms in $\cC_A$. We define $a_V^l$ as
\begin{align*}
a^l_V := a^r_V c_{A,V}\colon\ A\otimes V\isomorph V\otimes A \longrightarrow V.
\end{align*}
With this, $\big(V,a_V^l\big)$ is a \emph{left} module in $\cC$. As $A$ is commutative, the actions $a^r_V$, $a^l_V$ commute, $\big(V,a^l_V,a^r_V\big)$ becomes an $A$-bimodule in $\cC$, and $\Rep_{\cC}(A)$ is viewed as a full subcategory of $\bimod{A}(\cC)$ this way.

The category $\Rep_{\cC}(A)$ is monoidal as follows. Given two objects $V$, $W$ in $\Rep_{\cC}(A)$, their tensor product $V\otimes_A W$ is defined to be the coequalizer
\begin{align}\label{eq:rel-tensor}
\xymatrix{
V\otimes A \otimes W\ar@/^/[rr]^{a_V^r\otimes \ide_W}\ar@/_/[rr]_{\ide_V\otimes a_W^l}&& V\otimes W \ar[r]& V\otimes_A W,
}
\end{align}
which is an object in $\Rep_\cC(A)$ with the right $A$-module structure given by $a_{V\otimes_A W}^r = \ide_V\otimes a^r_W$. The unit object is the $A$-bimodule $A$ in $\cC$. This way, $\Rep_{\cC}(A)$ is a monoidal subcategory of $\bimod{A}(\cC)$.
\end{Definition}

\begin{Definition}[{\cite[Definition 2.1]{Par}}]
A right $A$-module $(V,a_V^r)$ in $\cC$ is called \emph{local} if
\begin{align*}
a^r_V&=a^r_V \; c_{A,V} \; c_{V,A}.
\end{align*}
The category of such local modules is denoted by $\locmod_{\cC}(A)$.
\end{Definition}

The category $\locmod_\cC(A)$ is a monoidal subcategory of $\Rep_\cC(A)$,
and $\locmod_\cC(A)$ is braided.
The braiding on $\locmod_\cC(A)$ is obtained from the braiding $c$ in $\cC$ which descends to the relative tensor products of two local modules. The algebra $A$ is \emph{trivializing} if $\locmod_\cC(A)\simeq \Vect$.

The definition of the monoidal category $\Rep_\cC(A)$ extends to the case when $\cD$ is a not necessarily braided monoidal category and $(A,c)$ is a commutative algebra in the Drinfeld center $\cZ(\cD)$ with half-braiding $c=\{c_X\mid X\otimes V\isomorph V\otimes X\}_{X\in \cD}$, see \cite[Section~4]{Sch}.

\begin{Definition}\label{def:RepA-general}
 Let $(A,c)$ be a commutative algebra in $\cZ(\cD)$. Define $\Rep_\cD(A,c)$ to be the category of right modules over $A$ in $\cD$ and monoidal structure given as in \eqref{eq:rel-tensor} with the left $A$-action defined by
 $a_{V}^l:=a_V^r c_V$ for $\big(V,a_V^r\big)$ a right $A$-module in $\cD$.
\end{Definition}
We will subsequently denote $\Rep_\cD(A,c)$ by $\Rep_{\cD}(A)$ when there is no confusion about which half-braiding is used.
We recall that by \cite[Corollary~4.5]{Sch}, the center $\cZ(\Rep_\cD(A))$ is equivalent to $\locmod_{\cZ(\cD)}(A)$ as a braided monoidal category. A special case of this result of interest occurs when~$\cC$ is already a braided monoidal category and we consider ${A^+=(A,c_{A,-})\in \ComAlg(\cZ(\cC))}$. Then Schauenburg's result gives an equivalence of braided monoidal categories between $\cZ(\Rep_\cC(A^+))$ and $\locmod_{\cZ(\cC)}(A^+)$.

Categories of local modules are a source of modular tensor categories, both in the semisimple case \cite[Theorem~4.5]{KO} and the non-semisimple case \cite[Theorem~4.12]{LW3}. For this, recall that a~rigid Frobenius algebra in a ribbon category $\cC$ is a connected commutative algebra satisfying the equivalent conditions of Proposition~\ref{prop:rigidFrobcharacterization}.

\begin{Theorem}\label{thm:locmodular}
If $\cC$ is a modular tensor category and $A$ is a rigid Frobenius algebra in $\cC$, then the category $\locmod_\cC(A)$ of local modules over $A$ in $\cC$ is also modular.
\end{Theorem}

Let $\cD$ be a finite tensor category and $(A,c)$ a rigid Frobenius algebra in $\cZ(\cD)$. The following lemma involves the left adjoint $U$ to the forgetful functor $\Rep_{\cD}(A)\to \cD$ and follows as in \cite[Lemma~4.5]{LW3}.
\begin{Lemma}\label{lem:U}
 The functor $U\colon \cD\to \Rep_\cD(A)$ which sends $X$ to $X\otimes A$ with right $A$-module structure given by multiplication in $A$ is a faithful dominant tensor functor.
\end{Lemma}

Powerful invariants of finite tensor categories are the \emph{Frobenius--Perron dimension} $\FPdim(\cD)$ and $\FPdim_\cD(X)$ for objects $X$ in $\cD$ \cite[Section~4.5]{EGNO}. The above Lemma~\ref{lem:U} implies that
\begin{equation}\label{eq:FPdim-Rep}
\FPdim \big(\Rep_{\cD}(A)\big)=\frac{\FPdim(\cD)}{\FPdim_\cD(A)}.
\end{equation}
This follows from \cite[Lemma~6.2.4]{EGNO} as in \cite[Lemma 4.5]{LW3}. Hence, as $\FPdim(\cZ(\cD))$ equals $\FPdim(\cD)^2$ by \cite[Theorem~7.16.6]{EGNO},
\[
\FPdim \big(\locmod_{\cZ(\cD)}(A)\big)=\frac{\FPdim(\cD)^2}{\FPdim_\cD(A)^2},
\]
see also \cite[Corollary~4.1]{DMNO}. If the category $\cD$ possesses a quasi-tensor functor $F$ to $\Vect$, then $\FPdim_{\cD}(X)=\dim_\Bbbk F(X)$ for any object $X$ in $\cD$ \cite[Proposition~4.5.7]{EGNO}.

\subsection{Dijkgraaf--Witten categories}
In this section, we give an explicit description of the Dijkgraaf--Witten (DW) categories $\cZ\big(\Vect_G^\omega\big)$ associated to a group $G$ and a $3$-cocycle $\omega$ of \cite{DPR} via the structure of twisted Yetter--Drinfeld modules following \cite[Proposition~3.2]{Maj98}.

Let $G$ be a finite group with a $3$-cocycle $\omega\in C^3(G,\Bbbk^\times)$, see equation~\eqref{3cocycle}.
Associated to this data, we define the category $\Vect_G^\omega$ of $G$-graded $\Bbbk$-vector spaces with associativity isomorphism given by \[\alpha((v_g\otimes w_h)\otimes u_k)=\omega^{-1}(g,h,k)v_g\otimes (w_h\otimes u_k),\]
where $v_g$, $w_h$, $u_k$ are $G$-homogeneous elements of degrees $g,h,k$, respectively. We observe that if $\phi\colon G\to G'$ is an isomorphism of groups such that $\omega$ and $\phi^{*}(\omega')$ define the same element in $H^3(G,\Bbbk^\times)$, then $\Vect_G^\omega$ and $\Vect_{G'}^{\omega'}$ (and hence their centers) are equivalent as monoidal categories (respectively, braided monoidal categories) \cite[Section~2.6]{EGNO}.

\begin{Lemma}\label{lem:omega-equiv}
If $\omega,\omega'\in C^3(G,\Bbbk^\times)$ are equivalent, then any choice of $\mu\in \Hom\big(G^2,\Bbbk^\times\big)$ such that $\deri(\mu)\omega'=\omega$ defines an equivalence of monoidal categories
\[T_\mu\colon\ \Vect^\omega_G\to \Vect_G^{\omega'},\]
which is the identity as a functor with monoidal structure given by
\[\mu^{T}_{\Bbbk_g,\Bbbk_h}=\mu(g,h)\ide_{\Bbbk_{gh}}\colon\ T_\mu(\Bbbk_g)\otimes T_\mu(\Bbbk_h)\to T_\mu(\Bbbk_{gh}),\]
where $\Bbbk_g$ denotes the $1$-dimensional vector space concentrated at degree $g$.
\end{Lemma}

Next, define the category of \emph{Yetter--Drinfeld (YD) modules} over $\Bbbk G$ twisted with respect to $\omega$. Such a twisted YD module has a $G$-grading (or, equivalently, a $\Bbbk G$-coaction)
\[V=\bigoplus\limits_{d \in G} V_d,\]
 a morphism
 \[a_V\colon\ \Bbbk G\otimes V\to V, \qquad g\otimes v\mapsto g\cdot v,\]
 which satisfies the twisted $\Bbbk G$-module condition that acting twice on the module is given by
\[
 h \cdot ( k \cdot v_d )=\tau (h,k)(d) hk \cdot v_d,
\]
where $\tau (h,k)(d)$ is defined in terms of the $3$-cocycle $\omega$ as follows:
\begin{align}\label{eq-taudef}
 \tau (h,k)(d):=\frac{\omega(h,k,d) \omega \big(hkd(hk)^{-1},h,k\big)}{\omega\big(h,kdk^{-1},k\big)}
\end{align}
and $v_d$ denotes a homogeneous element of $V$ of degree $d$.
The action and $G$-grading satisfy the \emph{YD compatibility condition} that action with $h \in G$ on the $d$-th component $v_d \in V_d$ will bring the component to the conjugated degree by $h$, namely $h \cdot v_d \in V_{h d h^{-1}}$. Morphisms of twisted YD modules are maps of $G$-graded $\Bbbk$-vector spaces $\phi$ that commute with the twisted actions in the sense that
$g\cdot \phi(v)=\phi(g\cdot v)$.

We note that the map $\tau$ can be derived from \cite[Section~1.3.3]{Wil} or \cite[Proposition 3.2]{Maj98} where right twisted modules are used. It can be interpreted as a $2$-cocycle on an appropriately defined groupoid \cite{Wil}.

There is a tensor product of twisted YD modules, which can be defined as the usual tensor product of graded vector spaces: given two such $V$ and $W$, the $d$-th graded component of $V\otimes W$, for $d\in G$, is given by
\begin{align*}
( V \otimes W )_d=\bigoplus\limits_{d=ab} V_a \otimes W_b.
\end{align*}
The module action will be given by
\[
h \cdot (v_d \otimes v_f )=\gamma(h)(d,f) (h \cdot v_d \otimes h \cdot v_f),
\]
with $d,f \in G$ and where
\begin{align}\label{eq-gammadef}
 \gamma (h)(d,f):=\frac{\omega(h,d,f) \omega\big(hdh^{-1},hfh^{-1},h\big)}{\omega\big(hdh^{-1},h,f\big)}.
\end{align}

We check that the tensor product of twisted YD modules is well defined.

 \begin{Lemma}
 The tensor product of two twisted YD modules is itself a twisted YD module.
 \end{Lemma}
 \begin{proof}
We need to check that
\[h \cdot ( k \cdot ( v_d \otimes v_f ) )=\tau (h,k ) ( df ) hk \cdot ( v_d \otimes v_f ).\]
This condition is equivalent to the equality
 \begin{align*}
 \gamma ( k ) ( d,f ) \gamma ( h ) \big( kdk^{-1},kfk^{-1} \big) \tau ( h,k ) ( d ) \tau ( h,k ) ( f ) =\tau ( h,k ) ( df ) \gamma ( hk ) ( d,f ),
 \end{align*}
 which is proven in Appendix \ref{appendix}, Lemma \ref{gammatau}.
 \end{proof}

 \begin{Lemma}
 The tensor product gives twisted YD modules the structure of a monoidal category.
 \end{Lemma}
 \begin{proof}
 The tensor product should be compatible with the monoidal structure morphisms (associator and unitors) of the category of twisted YD modules that we are working on.
 Since in this category the unitors are the identity this is clear, but for the case of the associator (following the opposite convention of \cite{EGNO}), given by:
$\alpha_{\Bbbk_g,\Bbbk_{g'},\Bbbk_{g''}}=\omega^{-1}(g,g',g'') \ide_{\Bbbk_g \otimes \Bbbk_{g'} \otimes \Bbbk_{g''}}$, we need to check that: $\alpha ( h \cdot ( [ \Bbbk_g \otimes \Bbbk_{g'} ] \otimes \Bbbk_{g''} ) )=h \cdot ( \alpha ( [ \Bbbk_g \otimes \Bbbk_{g'} ] \otimes \Bbbk_{g''} ) )$. Substituting the pertinent definitions, this equality amounts to:
\begin{gather*}
\omega^{-1}\!\big(hgh^{-1}\!,hg'h^{-1}\!,hg''h^{-1}\big) \gamma(h)(gg'\!,g'') \gamma(h)(g,g')=\gamma(h)(g,g'g'') \gamma(h)(g'\!,g'') \omega^{-1}(g,g'\!,g''),
\end{gather*}
which is proven in Appendix \ref{appendix}, Lemma \ref{eq-gammacond}. \end{proof}

\begin{Lemma}
 For two twisted YD modules $V$, $W$, there is a braiding given by
 \begin{align*}
 c_{V,W}\colon\ V \otimes W &\to W \otimes V, \\
 v_g \otimes w_h &\mapsto g \cdot w_h \otimes v_g.
 \end{align*}
\end{Lemma}
\begin{proof}
First, $c_{V,W}$ is a morphism of twisted $G$-modules by the identity
\begin{equation}
\gamma(k)(g,h)\tau\big(kgk^{-1},k\big)(h)=\gamma(k)\big(ghg^{-1},g\big)\tau(k,g)(h),\label{eq:gamma-braid}
\end{equation}
which holds by repeated use of the $3$-cocycle in Appendix \ref{3cocycle} with entries
\begin{itemize}\itemsep=0pt
 \item [--] $g_1=kghg^{-1}k^{-1}$, $g_2=k$, $g_3=gk^{-1}$, $g_4=k$,
 \item [--] $g_1=k$, $g_2=ghg^{-1}$, $g_3=gk^{-1}$, $g_4=k$,
 \item [--] $g_1=k$, $g_2=gk^{-1}$, $g_3=khk^{-1}$, $g_4=k$,
 \item [--] $g_1=k$, $g_2=gk^{-1}$, $g_3=k$, $g_4=h$.
\end{itemize}
The fact that $c_{V,W}^{-1}$ is also a morphism of twisted $G$-modules corresponds to the identity
\begin{equation}\label{eq:gamma-braid2}
 \gamma(k)(g,h)\tau\big(kh^{-1}k,k\big)(g)=\gamma(k)\big(h,h^{-1}gh\big)\tau\big(k,h^{-1}\big)(g).
\end{equation}

The braiding axioms are equivalent to the equalities;
\begin{align*}
 &\omega\big(g,hkh^{-1},h\big) = \omega(g,h,k)\omega\big(ghkh^{-1}g^{-1},g,h\big) \tau(g,h)(k)^{-1}, \\
 &\omega^{-1} \big(ghg^{-1}, g,k\big) = \omega^{-1}(g,h,k)\omega^{-1}\big(ghg^{-1},gkg^{-1},g\big) \gamma(g)(h,k)
\end{align*}
 both of which hold by \eqref{eq-taudef} and \eqref{eq-gammadef} respectively.
\end{proof}

The following proposition can be found, working with right twisted actions, in \cite[Section~3]{Maj98}.
\begin{Proposition}\label{prop:twistedYD-ZVect}
There is an equivalence of braided monoidal categories between $\cZ\big(\Vect_G^\omega\big)$ and the category of twisted YD modules over $G$ with respect to $\omega$, with reverse braiding.
\end{Proposition}
\begin{proof}
We only sketch the proof here. To an object $(V,c)$ in $\cZ\big(\Vect_G^\omega\big)$, one associates the morphism $a_V\colon \Bbbk G\otimes V\to V$ defined by $a_V=(\ide_V\otimes \varepsilon)c^{-1}_{\Bbbk G}$, where $\Bbbk G=\bigoplus_{g\in G} \Bbbk_g$ is the direct sum of all simple $G$-graded modules and $\varepsilon(g)=1$ for all $g\in G$. It follows from the tensor product compatibility of $c^{-1}$, that $a_V$ is a twisted $G$-action. Since $c^{-1}$ is a morphism in $\Vect_G^\omega$, the YD compatibility between coaction and twisted action follows. Conversely, a twisted $G$-action $a_V$ on~$V$ can be extend to an inverse half-braiding on an object $X$ in $\Vect_G^\omega$ by setting
\[c^{-1}_X(x_d\otimes v):=a_V(d,v)\otimes x_d \]
for all $x_d\in X_d$, $v\in V$. One verifies that these assignments extend to an equivalence of braided tensor categories.
\end{proof}

We require $\cZ\big(\Vect_G^\omega\big)$ to have a ribbon structure (which thus induces a pivotal structure). A~ribbon structure can be obtained from \cite[Theorem~5.4]{Shi2}. For this, we choose, in the notation of \cite[Section~3.2]{LW2}, the object $V=\one$, which is a square root of the distinguished invertible object $D=\one$ of $\Vect_G^\omega$, together with the identity $V\otimes V^{**}=D=\one$, and $\sigma_{V,X}\colon X\to X^{**}$ the monoidal natural isomorphism identifying an object with its double dual. This way, $\Vect_G^\omega$ is spherical and, hence, $\cZ\big(\Vect_G^\omega\big)$ is a ribbon category. This gives the following result.

\begin{Proposition}\label{prop:Zribbon}
For a twisted YD module $V$, define
\begin{align*}
\theta_V\colon\ V\to V, \qquad \theta_V(v_d)=d\cdot v_d,
\end{align*}
for $v_d\in V_d$. Then $\theta_V$ defines a ribbon twist which makes $\cZ\big(\Vect_G^\omega\big)$ a ribbon category.
\end{Proposition}
\begin{proof}
We have to check that $\theta$ satisfies \eqref{eq:ribbontwist-cond}. One checks that the first condition listed there is equivalent to the cocycle identity
\begin{align*}
 \gamma(df)(d,f)&=\tau\big(dfd^{-1},d\big)(d)\tau(d,f)(f).
\end{align*}
Using the definitions in equations \eqref{eq-taudef} and~\eqref{eq-gammadef}, this identity is equivalent to
the cocycle condition of $\omega$, cf.\ \eqref{3cocycle} with
$g_1=dfdf^{-1}d^{-1}$, $g_2=dfd^{-1}$, $g_3=d$, $g_4=f$. The condition~$\theta_\one=\ide_{\one}$ is clear, while $(\theta_V)^*=\theta_{V^*}$ follows from the fact that if $\{v_i\}$ is a homogeneous basis for~$V$ with dual basis~$\{f_i\}$, then $\{f_i\}$ is a homogeneous basis and $v_i\in V_g$ if and only if $f_i \in V_{g^{-1}}$.
\end{proof}

\section{Results}\label{sec:results}

Throughout the section, we assume that $\Bbbk$ is an algebraically closed field of arbitrary characteristic.

\subsection{Dijkgraaf--Witten categories of quotient groups via local modules}

In this section, we derive a general result on categories of local modules for algebras given by~$R(\one)$, which is a commutative algebra $(R(\one),c)$ in $\cZ(\cC)$ for $R$ the right adjoint of a tensor functor, using results of \cite{BN11,EGNO}. We specify this general result to DW categories of quotient groups.

\begin{Proposition}\label{prop:BN}
Let $\cC$, $\cD$ be tensor categories with a $\Bbbk$-linear, exact, monoidal functor $L \colon \cC \to \cD$ that has a right adjoint $R\colon \cD\to \cC$ which is faithful and exact. Then there is a tensor equivalence $\cD\simeq \Rep_\cC(R(\one),c)$ and a braided tensor equivalence between $\cZ(\cD)$ and $\locmod_{\cZ(\cC)}(R(\one))$.
\end{Proposition}
\begin{proof}
Set $A := R(\mathbbm{1})$. This is a commutative algebra in $\cZ(\cC)$.
By \cite[Proposition 6.1]{BN11}, there is a tensor equivalence $\cD \to \Rep_{\cC}(A)$ for the monoidal category from Definition~\ref{def:RepA-general}. By \cite[Corollary~4.5]{Sch}, we have that
$\cZ(\Rep_{\cC}(A))$ is tensor equivalent to $\locmod_{\cZ(\cC)}(A)$. Combining these two results yields the second claimed equivalence.
\end{proof}

We would like to be able to apply this result to the categories $\cC = \Vect_G^{\omega'}$ and $\cD = \Vect_H^\omega$. By~\cite[Section 2.6]{EGNO}, any monoidal functor $L\colon \Vect_G^{\omega'}\to \Vect_H^\omega$ corresponds to a group homomorphism $l\colon G\to H$ such that $\omega', l^*\omega$ are equal in $H^3(G,\Bbbk^\times)$, with respect to some $\mu\colon G\times G\to \Bbbk^\times$, i.e., $\omega'=(\deri \mu) l^*\omega$. Thus, on the simple objects of $\Vect_G^\omega$, the functor $L$ is given by $\Bbbk_g \mapsto \Bbbk_{l(g)}$.

If we combine this characterization with the adjunction condition, we get that a functor $R \colon \Vect_H^\omega \to \Vect_G^{\omega'}$ is right adjoint to $L$ if it satisfies, on the simple objects, that
\begin{equation*}
 \Hom_{\Vect_G^\omega}(\Bbbk_g,R(\Bbbk_h)) \cong \Hom_{\Vect_H^\omega}(L(\Bbbk_g),\Bbbk_h) \cong \Hom_{\Vect_H^\omega}(\Bbbk_{l(g)},\Bbbk_h).
\end{equation*}
The functor defined, for a an object $V=\bigoplus_{h\in H}V_h$, by
\[R(V)=\bigoplus_{g\in G}R(V)_{g}, \qquad \text{where}\quad R(V)_g := V_{l(g)} \quad \forall g\in G,
\] and, for a morphism $f\colon \bigoplus_h V_h\to \bigoplus_{h}W_h$ in $\Vect_H^\omega$, by
\[R(f)(v_g)= f(v_g)\in W_{l(g)}=R(W)_g, \qquad \text{for}\quad v_g\in V_{l(g)}=R(V)_g,\quad g\in G,\]
satisfies the above condition of being a right adjoint to $L$.

Assume that $V$ is a twisted YD module.
Now, $R(V)$ has the structure of a twisted YD module with action induced by the module action in $\cZ\big(\Vect_H^\omega\big)$, defined on $V_g=V_{l(g)}$, $g\in G$, by $k\cdot v_{l(g)} := l(k)\cdot v_{l(g)}$ for $v_{l(g)}\in V_g$ and $k\in G$, where on the right hand side $v_k\in V_{l(k)}$ is regarded as a vector in $R(V)$.

\begin{Lemma}
The functor $R\colon \Vect_H^{\omega} \to \Vect_G^{\omega'}$ is always exact and faithful if and only if $l\colon G\to H$ is surjective.
\end{Lemma}

In particular, we can apply the functor $R$ to obtain an algebra
\[A:= R(\mathbbm{1}) = \bigoplus\limits_{x \in \ker(l)}\Bbbk_x \] in $\Vect_G^{l^*\omega}$.
As a $\Bbbk$-vector space, A can be given a $\Bbbk$-basis $\{e_x \mid x \in \ker (l)\}$ with
\begin{itemize}\itemsep=0pt
 \item multiplication given by $e_xe_y = \mu(x,y) e_{xy}$,
 \item unit $1_A = e_1$.
\end{itemize}
By \cite[Proposition 6.1]{BN11}, $A$ is a commutative algebra in $\cZ\big(\Vect_G^\omega\big)$.
Proposition~\ref{prop:BN} now implies the following result.

\begin{Corollary}\label{cor:gen-equiv}
Given a surjective group homomorphism $l\colon G\!\to\! H$ and a $3$-cocycle ${\omega\in C^3(H,\Bbbk^\times)}$.
Then there is an equivalence of braided tensor categories $\cZ(\Vect_H^{\omega})\simeq \locmod_{\cZ(\Vect_G^{l^*\omega})}(A)$.
\end{Corollary}

\begin{Example}
Let $G$ be a group, $H$ be the trivial group $\{1\}$, with the cocycles $\omega$, $\omega'$ being trivial. Using the trivial group homomorphism and Corollary~\ref{cor:gen-equiv}, we get that
\[\Vect_\Bbbk=\cZ(\Vect_\Bbbk) \simeq \locmod_{\cZ(\Vect_G)}(A),\]
where $A = \Bbbk G$. Thus the group algebra $\Bbbk G \in \cZ(\Vect_G)$ is trivializing.
\end{Example}

\begin{Example}
Let us take $G$ to be an abelian group, with $H$ a subgroup. Then $H$ is isomorphic to some quotient of $G$, $H \cong G/N$. We can thus take $l \colon G \to H \cong G/N$ to be the quotient map, resulting in $A:= R(\mathbbm{1})$ being defined as an algebra in $\cZ(\Vect_G^{\omega})$, with $A_g =\Bbbk$ when $g \in N$, and the zero vector space otherwise, i.e., $A$ is the group algebra $\Bbbk N$. Thus ${\cZ(\Vect^\omega_{G/N}) \simeq \locmod_{\cZ(\Vect_G^{\omega'})}(\Bbbk N)}$.
\end{Example}

For a general subgroup $H$ of $G$, we cannot directly apply Corollary~\ref{cor:gen-equiv} as there may not be a~group homomorphism from $G$ to $H$.

\subsection{The induction functor}\label{sec:I}
Let $G$ be a group with a $3$-cocycle $\omega\in C^3(G,\Bbbk^\times)$. For a subgroup $H$ of $G$, we denote by $\omega|_H$ the restriction of $\omega$ to $H^3$. We denote the category of $H$-graded vector spaces twisted by $\omega|_H$ simply by $\Vect_H^\omega$.

In this section, we define a functor
\[\rI \colon\ \cZ( \Vect_H^\omega) \to \cZ\big(\Vect_G^\omega\big)\]
and show that this functor is Frobenius monoidal.
In objects, this functor is defined for any~$V \in \cZ( \Vect_H^\omega)$ as $V \mapsto \rI (V):=\mathbbm{k}G \otimes V$, with relations:
\begin{align}\label{eq:relationsIV}
gh \otimes v_d =\tau(g,h)(d)^{-1} g \otimes h \cdot v_d \qquad \text{for}\quad g \in G, \quad h,d \in H.
\end{align}
The $\Bbbk G$-coaction is given by
\begin{gather*}
\delta ( g \otimes v_d ) := gdg^{-1} \otimes g \otimes v_d.
\end{gather*}

\begin{Lemma}\label{lem:IV-twistedYD}
$\rI ( V )$ has the structure of a YD module with action: ${g\! \vartriangleright\! ( k \otimes v_d ):=\tau(g,k)(d) gk\! \otimes\! v_d}$.
\end{Lemma}

\begin{proof}
First, we check that the coaction is compatible with the comultiplication, namely that $ ( \Delta \otimes \ide ) \delta = (\ide \otimes \delta ) \delta$. Here,{\samepage
 \begin{gather*}
 ( \Delta \otimes \ide ) \delta ( g \otimes v_d ) = ( \Delta \otimes \ide ) ( gdg^{-1} \otimes g \otimes v_d ) = gdg^{-1} \otimes gdg^{-1} \otimes g \otimes v_d, \\
 ( \ide \otimes \delta ) \delta ( g \otimes v_d ) = gdg^{-1} \otimes ( g \otimes v_d ) = gdg^{-1} \otimes gdg^{-1} \otimes g \otimes v_d.
 \end{gather*}
Both sides agree.}

Next, the YD condition requires that for any $g, k \in G$, $d \in H$, the composition
 \begin{align*}
 g \otimes ( k \otimes v_d ) & \mapsto ( g \otimes g ) \otimes ( k \otimes v_d )\mapsto \tau ( g,k ) ( d ) g \otimes ( gk \otimes v_d ) \\ & \mapsto \tau ( g,k ) ( d ) g \otimes gkdk^{-1}g^{-1} \otimes ( gk \otimes v_d )
 \\&\mapsto \tau ( g,k ) ( d ) gkdk^{-1} \otimes ( gk \otimes v_d )
 \end{align*}
 needs to be equal to
 \begin{align*}
 g \otimes ( k \otimes v_d ) & \mapsto ( g \otimes g ) \otimes kdk^{-1} \otimes ( k \otimes v_d ) \mapsto \tau ( g,k ) ( d ) gkdk^{-1} \otimes ( gk \otimes v_d ),
 \end{align*}
which match.

For the defined map to be a twisted action, we require that, for any $ g,h,k \in G$, $d \in H$:
 \begin{equation*}
 g \vartriangleright (h \vartriangleright ( k \otimes v_d )) = \tau(g,h)\big(kdk^{-1}\big) gh \vartriangleright (k \otimes v_d ).
 \end{equation*}
By expanding both sides using the proposed action, this is equivalent to requiring that
 \begin{equation*}
 \tau(h,k)(d)\tau(g,hk)(d)ghk \otimes v_d = \tau(g,h)(kdk^{-1})\tau(gh,k)(d)ghk \otimes v_d.
 \end{equation*}
This equality holds by Lemma~\ref{eq-taucond}.
\end{proof}

We claim that the functor $I$ defined above is \emph{op-lax} monoidal with natural transformation
\begin{align*}
 \nu_{V,W} \colon\ \rI ( V \otimes V ) & \to \rI ( V ) \otimes \rI ( W ), \nonumber \\
 g \otimes ( v_d \otimes w_f ) & \mapsto \gamma ( g ) ( d,f ) ( g \otimes v_d ) \otimes ( g \otimes w_f ),
\end{align*}
with $\gamma$ as defined in \eqref{eq-gammadef}, for any $d,f \in H$, and $V,W \in \Ob ( \cZ( \Vect_H^\omega))$. The counit is given by
\begin{equation*}
 \rI(\one)\to \one, \qquad g\otimes 1\mapsto 1.
\end{equation*}

\begin{Lemma}\label{lem:oplax}
The natural transformation $\nu$ equips the functor $\rI$ with an op-lax monoidal structure.
\end{Lemma}

\begin{proof}
Let $k \in G$ and $V,W \in \Ob \big( \cZ\big(\Vect_H^\omega\big) \big)$. For $\nu_{V,W}$ to be a morphism of YD modules, we require that \[\nu_{V,W}(k\vartriangleright (g \otimes (v_d \otimes w_f))) = k \vartriangleright (\nu_{V,W}(g \otimes (v_d \otimes w_f))).\]
We compute that
\begin{gather*}
 \nu_{V,W}(k\vartriangleright (g \otimes (v_d \otimes w_f))) = \gamma(kg)(d,f)\tau(k,g)(df)(kg \otimes v_d) \otimes (kg \otimes w_f), \\
 k \vartriangleright (\nu_{V,W}(g \otimes (v_d \otimes w_f))) = \gamma(g)(d,f) k \vartriangleright ((g\otimes v_d) \otimes (g \otimes w_f)) \\
\qquad{} =\gamma(k)\big(gdg^{-1},gfg^{-1}\big)\gamma(g)(d,f)
 (k\vartriangleright(g \otimes v_d) \otimes k \vartriangleright (g \otimes w_f)) \\
 \qquad{} =\tau(k,g)(d)\tau(k,g)(f)\gamma(k)\big(gdg^{-1},gfg^{-1}\big) \gamma(g)(d,f)(kg \otimes v_d) \otimes (kg \otimes w_f).
\end{gather*}
The two expressions are equal by Lemma~\ref{gammatau}, so $\nu_{V,W}$ is indeed a morphism of twisted YD modules.

To check the conditions of an op-lax monoidal functor, we need to verify that, for any objects $V,W,U \in \Ob \big( \cZ \big( \Vect_H^\omega \big)\big) $, the following morphisms \[
\rI( ( V \otimes W ) \otimes U) \to \rI(V) \otimes (\rI(W) \otimes \rI(U))\]
 are equal
\begin{equation*}
 (\ide_{\rI(V)}\otimes \nu_{W,U})\nu_{V,W\otimes U}I(\alpha_{V,W,U}) = \alpha_{\rI(V),\rI(W),\rI(U)} (\nu_{V,W} \otimes \ide_{\rI(U)}) \nu_{V\otimes W,U}.
\end{equation*}
Evaluating on a vector $(g \otimes ((v_d \otimes w_f) \otimes u_h))$ in $\rI( ( V \otimes W ) \otimes U)$, we compute
\begin{align*}
 &(\ide_{\rI(V)}\otimes \nu_{W,U})\nu_{V,W\otimes U}I(\alpha_{V,W,U})(g \otimes ((v_d \otimes w_f) \otimes u_h)) \\
 &\quad= \gamma(g)(f,h)\gamma(g)(d,fh) \omega^{-1}(d,f,h)((g\otimes v_d)\otimes ((g\otimes w_f) \otimes (g\otimes u_h))),\\
&\alpha_{\rI(V),\rI(W),\rI(U)} (\nu_{V,W} \otimes \ide_{\rI(U)}) \nu_{V\otimes W,U}(g \otimes ((v_d \otimes w_f) \otimes u_h)) \\
 &\quad= \omega^{-1}\big(gdg^{-1},gfg^{-1},ghg^{-1}\big)\gamma(g)(d,f)\gamma(g)(df,h)
((g\otimes v_d)\otimes ((g\otimes w_f) \otimes (g\otimes u_h))).
\end{align*}
These two expressions are equal by Appendix~\ref{eq-gammacond}.

The unitality condition is easily verified, using $\gamma(g)(d,1)=\gamma(g)(1,d)=1$.
\end{proof}

The functor $\rI$ is also \emph{lax} monoidal with the structural natural transformation
\begin{align*}
 \mu_{V,W}\colon\ \rI(V)\otimes \rI(W)&\to \rI(V\otimes W).
\end{align*}
This map $\mu_{V,W}$ sends a vector $(g\otimes v_d)\otimes (k\otimes w_f)$ to zero unless $g^{-1}k\in H$. If $gH=kH$ we can use equation~\eqref{eq:relationsIV} to replace
$(g\otimes v_d)\otimes (k\otimes w_f)$ by a vector of the form $(g\otimes v_d)\otimes (g\otimes w')$, with~$w'$ having degree $g^{-1}kfk^{-1}g$. Hence,
it suffices to describe the image of $\mu_{V,W}$ on vectors of the form $(g\otimes v_d)\otimes (g\otimes w_f)$, with $g\in G$ and $v_d$, $w_f$ homogeneous vectors of degrees $d$ in $V$ and~$f$ in~$W$, respectively,
\[
 \mu_{V,W}((g\otimes v_d)\otimes (g\otimes w_f))=\gamma(g)(d,f)^{-1}g\otimes (v_d\otimes w_f).
\]

In addition, we define the unit of this lax monoidal structure by
\begin{equation}\label{eq:lax-unit}
u\colon\ \one \to I(\one), \qquad 1\mapsto \sum_{i}g_i,
\end{equation}
where $\{g_i\}_{i\in I}$ is a set of representatives for the left cosets of $H$ in $G$, i.e., $G=\coprod\limits_i g_iH$.

\begin{Lemma}\label{lem:lax}
The natural transformation $\mu$ equips the functor $\rI$ with a lax monoidal structure.
\end{Lemma}
\begin{proof}
First, we check that $\mu_{V,W}$ is a well-defined morphism of twisted YD modules over $G$. For $k,g \in G$ and $V,W \in \Ob \big( \cZ\big(\Vect_H^\omega\big) \big)$, we require that \[\mu_{V,W}(k\vartriangleright ((g \otimes v_d) \otimes (g \otimes w_f))) = k \vartriangleright (\nu_{V,W}((g \otimes v_d) \otimes (g \otimes w_f))).\]
We compute that
\begin{gather*}
 k \vartriangleright (\nu_{V,W}((g \otimes v_d) \otimes (g \otimes w_f))) = \tau(k,g)(df)\gamma(g)(d,f)^{-1}kg \otimes (v_d \otimes w_f),\\
 \mu_{V,W}(k\vartriangleright ((g \otimes v_d) \otimes (g \otimes w_f))) \\
 \qquad{} = \gamma(kg)(d,f)^{-1}\tau(k,g)(d)\tau(k,g)(f) \gamma(k)\big(gdg^{-1},gfg^{-1}\big) kg \otimes (v_d \otimes w_f).
\end{gather*}
These two expressions are equal by Lemma~\ref{gammatau}, so $\mu_{V,W}$ is a morphism of twisted YD modules.

Now we have to check the defining diagrams of a lax monoidal structure. Firstly, consider three objects $V$, $W$, $U$ in $\cZ\big(\Vect_H^\omega\big)$. We need to check that the following morphisms
\[(\rI(V)\otimes \rI(W))\otimes \rI(W)\to \rI(V\otimes (W\otimes U))\] are equal
\begin{align*}
 \mu_{V,W\otimes U}(\ide_{\rI(V)}\otimes \mu_{W,U})\alpha_{\rI(V),\rI(W),\rI(U)}=\rI(\alpha_{V,W,U})\mu_{V\otimes W,U}(\mu_{V,W}\otimes \ide_{\rI(W)}).
\end{align*}
It suffices to check this for vectors of the form $((g\otimes v_d)\otimes (g\otimes w_f))\otimes (g\otimes u_h)$.
We compute
\begin{align*}
 &\mu_{V,W\otimes U}(\ide_{\rI(V)}\otimes \mu_{W,U})\alpha_{\rI(V),\rI(W),\rI(U)}(((g\otimes v_d)\otimes (g\otimes w_f))\otimes (g \otimes u_h))\\
 &\qquad =
\gamma^{-1}(g)(d, fh)\gamma^{-1}(g)(f,h)\omega^{-1}\big(gdg^{-1},gfg^{-1},ghg^{-1}\big)g\otimes (v_d\otimes (w_f\otimes u_h)),
 \\
&\rI(\alpha_{V,W,U})\mu_{V\otimes W,U}(\mu_{V,W}\otimes \ide_{\rI(W)})(((g\otimes v_d)\otimes (g\otimes w_f))\otimes (g \otimes u_h))\\
&\qquad = \gamma^{-1}(g)(d,f)\gamma^{-1}(g)(df,h)\omega^{-1}(d,f,h)g\otimes (v_d\otimes (w_f\otimes u_h)).
\end{align*}
The two expressions are equal by Lemma~\ref{eq-gammacond}.

The unitality conditions for $\rI$ are easily verified, using $\gamma(g)(d,1)=\gamma(g)(1,d)=1$.
\end{proof}

\begin{Proposition}\label{prop:Frob}
The functor $\rI\colon \cZ\big(\Vect_H^\omega\big)\to \cZ\big(\Vect_G^\omega\big)$ is a separable Frobenius monoidal functor.
\end{Proposition}
\begin{proof}
The claim that $\rI$ is a Frobenius monoidal functor follows from checking the diagrams in Definition~\ref{def:Frobmonoidal}. This can be tested on vectors of the form $g_i\otimes v_d\in \rI(V)$, where $\{g_i\}$ is a set of left coset representatives. Commutativity of both diagrams amounts to the condition in Lemma~\ref{eq-gammacond}. Finally, $\rI$ is a separable Frobenius monoidal functor as, clearly, $\mu_{V,W}\nu_{V,W}=\ide_{\rI(V\otimes W)}$.
\end{proof}

\begin{Proposition}\label{prop:braidedFrob}
The functor $I$ is both a braided lax monoidal and braided oplax monoidal functor and preserves the ribbon structure.
\end{Proposition}
\begin{proof}
We start by checking that the lax monoidal structure given by $\mu$ is compatible with the braiding. First, we need to check that we can restrict to vectors of the form $ ( g \otimes v_d ) \otimes (g \otimes w_f)$.

Consider the composition $\rI(c_{V,W})\mu_{V,W}$. By our earlier discussion, this is zero on all vectors not of the proposed form. For $\mu_{W,V}c_{\rI(V),\rI(W)}$ on a generic vector in $\rI(V) \otimes \rI(W)$, we get that
\begin{align*}
\mu_{W,V}c_{\rI(V),\rI(W)}(( g \otimes v_d ) \otimes (k \otimes w_f)) &= \mu_{W,V} \big(\big( gdg^{-1} \vartriangleright (k \otimes w_f)\big) \otimes (g \otimes v_d)\big) \\
 &= \tau\big(gdg^{-1},k\big)(f) \mu_{W,V} \big(\big( gdg^{-1}k \otimes w_f\big)\big) \otimes (g \otimes v_d).
\end{align*}
Now this term is non-zero only when $kgd^{-1} \in H$, which is equivalent to requiring $g^{-1}k \in H$. Hence we can restrict to vector to be of the proposed form and compute
\begin{gather*}
 \mu_{W,V}c_{\rI(V),\rI(W)}(( g \otimes v_d ) \otimes (g \otimes w_f)) \\
 \qquad{} = \mu_{W,V} \big(\big( gdg^{-1} \vartriangleright (g \otimes w_f)\big) \otimes (g \otimes v_d)\big) \\
 \qquad{} = \tau\big(gdg^{-1},g\big)(f) \mu_{W,V} ((gd \otimes w_f) \otimes (g \otimes v_d)) \\
\qquad{} = \tau\big(gdg^{-1},g\big)(f)\tau(g,d)(f)^{-1} \mu_{W,V} ((g \otimes d \cdot w_f) \otimes (g \otimes v_d))\\
\qquad{}{}=\tau\big(gdg^{-1},g\big)(f)\tau(g,d)(f)^{-1} \gamma(g)\big(dfd^{-1},d\big)^{-1} g \otimes (d \cdot w_f \otimes v_d),\\
 \rI(c_{V,W})\mu_{V,W} (( g \otimes v_d ) \otimes (g \otimes w_f)) \\
 \qquad{} = \gamma(g)(d,f)^{-1} \rI(c_{V,W}) (g \otimes (v_d \otimes w_f)) = \gamma(g)(d,f)^{-1} g \otimes (d \cdot w_f \otimes v_d).
\end{gather*}
The two expressions are equal using \eqref{eq-taudef} and~\eqref{eq-gammadef}.

The braided oplax monoidal condition follows similarly. We compute that
\begin{gather*}
 \nu_{W,V}\rI(c_{V,W})(g \otimes (v_d \otimes w_f)) = \nu_{W,V}(g \otimes (d \vartriangleright w_f \otimes v_d)) \\
\phantom{ \nu_{W,V}\rI(c_{V,W})(g \otimes (v_d \otimes w_f))}{} = \gamma(g)(dfd^{-1},d)((g\otimes d \vartriangleright w_f) \otimes (g \otimes v_d)), \\
 c_{\rI(V),\rI(W)}\nu_{V,W} (g \otimes (v_d \otimes w_f)) = \gamma(g)(d,f) c_{\rI(V),\rI(W)} ((g \otimes v_d) \otimes (g \otimes w_f)) \\
\phantom{ c_{\rI(V),\rI(W)}\nu_{V,W} (g \otimes (v_d \otimes w_f)) }{} = \gamma(g)(d,f) ((gdg^{-1} \vartriangleright (g \otimes w_f)) \otimes (g \otimes v_d)) \\
 \phantom{ c_{\rI(V),\rI(W)}\nu_{V,W} (g \otimes (v_d \otimes w_f)) }{}= \tau(gdg^{-1},g)(f) \gamma(g)(d,f) ((gd \otimes w_f)) \otimes (g \otimes v_d)) \\
\phantom{ c_{\rI(V),\rI(W)}\nu_{V,W} (g \otimes (v_d \otimes w_f)) }{} = \tau(g,d)(f)^{-1}\tau(gdg^{-1},g)(f) \gamma(g)(d,f)\\
\phantom{ c_{\rI(V),\rI(W)}\nu_{V,W} (g \otimes (v_d \otimes w_f)) }{}\quad\ \cdot((g\otimes d \vartriangleright w_f) \otimes (g \otimes v_d)).
\end{gather*}
Again, these expressions are equal by equations \eqref{eq-taudef} and \eqref{eq-gammadef}.

Further, $I(\theta_V)=\theta_{I(V)}$ with the ribbon structure defined in Remark~\ref{rem:ribbon}.
 Indeed, we compute that
 \begin{align*}
 \theta_{I(V)}(g\otimes v_d)&=gdg^{-1}\vartriangleright (g\otimes v_d)
 =\tau(gdg^{-1},g)(d)gd\otimes v_d\\
 &=\tau(gdg^{-1},g)(d)\tau(g,d)(d)^{-1} g\otimes d\cdot v_d
 =g\otimes d\cdot v_d=I(\theta_V)(g\otimes v_d),
 \end{align*}
 where the second-to-last equality uses equation~\eqref{eq:gamma-braid} with $k=g$, $g=d$, $h=d$.
\end{proof}

\begin{Corollary}\label{cor:Ipreserves}
If $A$ is an algebra $($respectively, coalgebra or Frobenius algebra$)$ in $\cZ( \Vect_H^\omega)$, then $\rI(A)$ is an algebra $($respectively, coalgebra or Frobenius algebra$)$ in $\cZ( \Vect_G^\omega)$. Moreover, if $A$ is commutative $($respectively, cocommutative$)$ in $\cZ( \Vect_H^\omega)$, then $\rI(A)$ is commutative $($respectively, cocommutative$)$ in $\cZ( \Vect_G^\omega)$.
\end{Corollary}

\begin{Example}\label{ex:AH}
The tensor unit $\one$ is a commutative and cocommutative Frobenius algebra in~$\cZ( \Vect_H^\omega)$. Hence, $\rI(\one):=A_H$ inherits these properties. Explicitly, $A_H$ is spanned as a~$\Bbbk$-vector space by $\Set{\delta_{gH}\mid g\in G }$ subject to the relations that $\delta_{gH}=\delta_{kH}$ if and only if $g^{-1}k\in H$. Further, $A_H$ is a twisted YD module via
\[
k\cdot \delta_{gH}=\delta_{kgH}, \qquad \delta(\delta_{gH})=1\otimes \delta_{gH}.
\]
The multiplication and unit are given by
\[\delta_{gH}\delta_{kH}=\begin{cases}\delta_{gH} & \text{if $g^{-1}k\in H$},\\ 0 & \text{otherwise},\end{cases}
\qquad 1_{A_H}=\sum_{i}\delta_{g_iH},\]
for $\Set{g_i}$ a set of $H$-coset representatives. The comultiplication and counit
\[\Delta_{A_H}(\delta_{gH})=\delta_{gH}\otimes \delta_{gH}, \qquad \varepsilon_{A_H}(\delta_{gH})=1\]
make $A_H$ a commutative and cocommutative Frobenius algebra in $\cZ( \Vect_G^\omega)$. We note that, since $A_H$ has trivial $G$-grading, the braiding is simply given by
$a\otimes b\mapsto b\otimes a$, for all $a,b\in A_H$.
\end{Example}

Consider the category $\Rep_{\Vect_G^\omega}(A_H)$ from Definition~\ref{def:RepA-general}. We fix a set of coset representatives~$\Set{g_i}$ of $H$ in $G$ such that $g_1=1$ and denote the corresponding basis of $A_H$ by $\Set{\delta_{g_i}}$. We can now define a functor
\[T\colon\ \Vect_H^\omega\to \Rep_{\Vect_G^\omega}(A_H), \qquad T(V)= A_H\otimes V,\qquad T(f)=\ide_{A_H}\otimes f,\]
where $A_H\otimes V$ is a $\Bbbk G$-comodule via the coaction
\[\delta(\delta_{g_i}\otimes v)=g_i|v|g_i^{-1}\otimes (\delta_{g_i}\otimes v),\]
and a right $A_H$-module via
\[(\delta_i\otimes v)\cdot \delta_j= \delta_{i,j}(\delta_i\otimes v).\]
Next, consider the canonical isomorphisms $\mu^T_{V,W}$ appearing in
\[(A_H\otimes V)\otimes (A_H\otimes W)\to (A_H\otimes V)\otimes_{A_H} (A_H\otimes W)\xrightarrow{\mu^T_{V,W}} A_H\otimes (V\otimes W),\]
which is given by
\[\mu^T_{V,W}((\delta_{g_i}\otimes v)\otimes (\delta_{g_i}\otimes w))=\gamma(g_i)(|v|,|w|)^{-1}\delta_{i,j}\delta_{g_i}\otimes (v\otimes w).\]
\begin{Lemma}\label{lem:T-monoidal}
 The functor $T$ is monoidal.
\end{Lemma}
\begin{proof}
First, we check that $\mu^T_{V,W}$ is obtained as factorization over the relative tensor product~$\otimes_{A_H}$ as stated and becomes an isomorphism. Further, the coherence diagram making $T$ a~monoidal functor follows from Lemma~\ref{eq-gammacond}.
\end{proof}

We will see in Proposition~\ref{prop:RepAH-equiv} below that this functor gives an equivalence of tensor categories when $|G:H|\neq 0$.

\subsection{Local modules over coset algebras}

In this section, we prove that the functor $\rI$ from Section~\ref{sec:I} induces an equivalence of braided monoidal categories between
$\cZ( \Vect_H^\omega)$ and $\locmod_{\cZ( \Vect_G^\omega)}(A_H)$, where
$A_H=\rI(\one)\cong \Bbbk(G/H)$ is the algebra of functions on left cosets of $H$ in $G$.

\begin{Lemma}\label{lem:AH-action}
For any object $V$ in $\cZ\big(\Vect_H^\omega\big)$, $\rI(V)$ is a right local module over the algebra $A_H$ from Example~{\rm \ref{ex:AH}}. The right action is given by
\[
a_{\rI(V)}^r\colon\ \rI(V)\otimes A_H\to \rI(V), \qquad (g\otimes v)\cdot \delta_{kH}=
\begin{cases}
g\otimes v & \text{if $k^{-1}g\in H$,}\\
0 & \text{otherwise.}
\end{cases}
\]
\end{Lemma}
\begin{proof}
The fact that $\rI(V)$ is a right $A_H=\rI(\one)$-module follows from Proposition~\ref{prop:Frob-perserved}\,(c) and Lemma~\ref{lem:lax}.

We check that $\rI(V)$ is a local module. By applying the braiding twice, we obtain
\[
c_{A_H,\rI(V)}c_{\rI(V),A_H}((g\otimes v_d)\otimes \delta_{kH}) = c_{A_H,\rI(V)}(\delta_{gdg^{-1}kH} \otimes (g\otimes v_d)) = (g\otimes v_d) \otimes \delta_{gdg^{-1}kH}.
\]
Applying the right action gives us
\[
(g\otimes v_d) \otimes \delta_{gdg^{-1}kH}=
\begin{cases}
g\otimes v_d & \text{if $k^{-1}gd^{-1}\in H$,}\\
0 & \text{otherwise.}
\end{cases}
\]
As $d\in H$, this is exactly the result of applying the right action only.
\end{proof}

The following result is independent of the choice of ribbon structure for $\cZ\big(\Vect_G^\omega\big)$ and we may use the ribbon structure from Proposition~\ref{prop:Zribbon}.
\begin{Lemma}\label{lem:AH-rigidFrob}
Assume $|G:H|\in \Bbbk^\times$. Then the algebra $A_H$ from Example~{\rm \ref{ex:AH}} is a rigid Frobenius algebra in $\cZ\big(\Vect_G^\omega\big)$.
\end{Lemma}
\begin{proof}
The trivial algebra $\one$ is a commutative Frobenius algebra in $\cZ\big(\Vect_H^\omega\big)$. Hence, ${A_H\!=\!R(\one)}$ is a commutative Frobenius algebra in $\cZ\big(\Vect_G^\omega\big)$ by Corollary~\ref{cor:Ipreserves}. The comultiplication and counit are given by
\[\Delta_{A_H}(\delta_{gH})=\delta_{gH}\otimes \delta_{gH}, \qquad \varepsilon_{A_H}(\delta_{gH})=1.\]
We first check that $A_H$ is a connected algebra. Indeed, as $A_H$ is concentrated in $G$-degree $1$, it is a $G$-module and
\smash{$\Hom_{\cZ(\Vect_G^\omega)}(\one,A_H)\subseteq (A_H)^G$}. The latter space of $G$-invariant elements in $A_H$ is one-dimensional since $A_H$ is given by functions on a transitive $G$-set.
Now, we compute that
\begin{align*}
 m_{A_H} \Delta_{A_H}(\delta_{gH})=\delta_{gH}\delta_{gH}=\delta_{gH}, \qquad \varepsilon_{A_H}(1_{A_H})=|G:H|.
\end{align*}
Since, by assumption, $|G:H|\in \Bbbk^\times$, $A_H$ is a rigid Frobenius algebra of dimension $\dim_j(A_H)=|G:H|$, cf.\ Proposition~\ref{prop:rigidFrobcharacterization}.
\end{proof}

\begin{Proposition}\label{prop:RepAH-equiv}
Assume $|G:H|\in \Bbbk^\times$. The functor $T$ from Lemma~{\rm \ref{lem:T-monoidal}} induces an equivalence of tensor categories from $\Vect_H^\omega$ to $\Rep_{\Vect_G^\omega}(A_H)$.
\end{Proposition}
\begin{proof}
We first check that $T$ is fully faithful. For this, we note that every object $X$ in $\Rep_{\Vect_G^\omega}(A_H)$ has a direct sum decomposition
$X=\oplus_iX^i$, where $X^i$ is the image of the action of the idempotent $\delta_{g_i}\in A_H$. Any morphism of $A_H$-modules preserves this direct sum decomposition. Thus, for given objects $V$, $W$ in $\Vect_H^\omega$, a morphism $f\colon T(V)\to T(W)$ and any $v\in V$,
\[f(g_i\otimes v)=g_i\otimes g(v),\]
for a unique vector $g(v)\in W$. The mapping $g\colon V\to W$ preserves the $H$-grading since conjugation by $g_i$ is bijective. Thus, $f=T(g)$ and $T$ is full. Further, $T$ is faithful as the tensor product~${(-)\otimes_\Bbbk A_H}$ is faithful.

Now, $\Vect_G^\omega$ is a finite tensor category and, as $A_H$ is a rigid Frobenius algebra by Lemma~\ref{lem:AH-rigidFrob} provided that $|G:H|\neq 0$, $\Rep_{\Vect_G^\omega}(A_H)$ is a finite tensor category by \cite[Corollary 4.21]{LW3}. We conclude that $T\colon \Vect_H^\omega\to \Rep_{\Vect_G^\omega}(A_H)$ is a fully faithful tensor functor. Now, equation~\eqref{eq:FPdim-Rep} specifies to
\[\FPdim\big(\Rep_{\Vect_G^\omega}(A_H)\big)=\frac{\FPdim\big(\Vect_G^\omega\big)}{\FPdim(A_H)}=\frac{|G|}{\dim_\Bbbk(A_H)}=\frac{|G|}{\frac{|G|}{|H|}}=|H|=\FPdim\big(\Vect_H^\omega\big).\]
Thus, \cite[Proposition~6.3.3]{EGNO} implies that $T$ gives an equivalence.
\end{proof}

Next, we will extend the equivalence of $\Vect_H^\omega$ and $\Rep_{\Vect_G^\omega}(A_H)$ to Drinfeld centers and local modules using the functor $\rI$ from Section~\ref{sec:I}.

\begin{Lemma}\label{lem:Imonoidal}
The functor $\rI$ induces a monoidal functor
from $\cZ\big(\Vect_H^\omega\big)$ to $\locmod_{\cZ(\Vect_G^\omega)}(A_H)$, the category of local modules over $A_H$. This functor is a ribbon functor if $|G:H|\in \Bbbk^\times$.
\end{Lemma}
\begin{proof}
By Lemma~\ref{lem:AH-action}, it is clear that $\rI$ induces a functor $\rI \colon \cZ\big(\Vect_H^\omega\big) \to \locmod_{\cZ(\Vect_G^\omega)}(A_H)$.
The tensor product of $X,Y \in \locmod_{\cZ(\Vect_G^\omega)}(A_H)$ is the relative tensor product $X\otimes_{A_H} Y$ defined in equation \eqref{eq:rel-tensor}, with the left action given by $a_W^l := a^r_W c_{A_H,W}$.
In $\rI(U) \otimes_{A_H} \rI(V)$, this gives us that $(g \otimes u_d) \otimes (l \otimes v_e) = 0 \text{ if and only if } gH \neq lH$, for non-zero $u_d$, $v_e$.

This suggests that \[\rI(U) \otimes_{A_H} \rI(V) = \rI(U) \otimes \rI(V)/S \qquad\text{with}\quad S = \mathrm{span}_\Bbbk\{(g\otimes u_d) \otimes (l\otimes v_f) \mid l^{-1}g \notin H\}.\]

In order to be compatible with the twisted YD module and local module structure, we need~$S$ to be a subobject in \smash{$\locmod_{\cZ(\Vect_G^\omega)}(A_H)$}.
To see this, we first note that $S$ is $G$-graded as \[\delta((g\otimes u_d) \otimes (l\otimes v_f)) = gdg^{-1}lfl^{-1} \otimes \big( (g\otimes u_d) \otimes (l\otimes v_f) \big)\] gives a $G$-homogeneous spanning set. Secondly, $S$ is closed under the twisted $G$-action. This follows as
\begin{gather*}
k \vartriangleright ((g\otimes u_d) \otimes (l\otimes v_f)) = \gamma(k)\big(gdg^{-1},lfl^{-1}\big)\tau(k,g)(d)\tau(k,l)(f)(kg\otimes u_d) \otimes (kl\otimes v_f) \!\in \! S
\end{gather*}
because $(kl)^{-1}(kg) \notin H$ if and only if $l^{-1}g\notin H$.
Finally, $S$ is closed under the right action of~$A_H$ since\looseness=1
\[((g\otimes u_d) \otimes (l\otimes v_f) ) \cdot \delta_{kh} =
\begin{cases}
(g\otimes u_d) \otimes (l\otimes v_f) & \text{if }kH = lH,\\
0 & \text{else},
\end{cases}
\]
is clearly in $S$. Hence $S$ is a subobject.

As a consequence of this quotient, the op-lax monoidal structure from Lemma \ref{lem:oplax} extends to\looseness=-1
\begin{align*}
 \bar{\nu}_{U,V}\colon\ \rI(U\otimes V) &\longrightarrow \rI(U) \otimes_{A_H} \rI(V), \\
 g \otimes (u_d \otimes u_f) &\mapsto \gamma(g)(d,f) (g \otimes u_d) \otimes (g \otimes v_f),
\end{align*}
also giving an op-lax monoidal structure.

The natural transformation $\ov{\nu}$ is in fact an isomorphism. To observe this, we define the morphism
\begin{align*}
 \Lambda_{U,V} \colon\ \rI(U) \otimes \rI(V) &\longrightarrow \rI(U \otimes V), \\
 (g \otimes u_d) \otimes (l \otimes v_f) &\mapsto
 \begin{cases}
 \lambda(g,d,l,f) g \otimes \big(u_d \otimes g^{-1}l \cdot v_f\big) & \text{if } gH=lH,\\
 0 & \text{else},
 \end{cases}
\end{align*}
where
\[\lambda(g,d,l,f) =\frac{1}{ \tau\big(g,g^{-1}l\big)(f)\gamma(g)\big(d,g^{-1}lfl^{-1}g\big)}. \]
For this morphism to be a morphism of YD modules, it needs the equality
\[
\lambda(g,d,l,f)\tau(k,g)\big(dg^{-1}lfl^{-1}g\big) = \gamma(k)\big(gdg^{-1},lfl^{-1}\big)\tau(k,g)(d)\tau(k,l)(f)\lambda(kg,d,kl,f)
\]
to hold, which follows straightforwardly from Lemmas~\ref{eq-taucond} and \ref{gammatau}. As the kernel of this morphism $\Lambda_{U,V}$ contains $S$, it induces a quotient morphism $\bar{\Lambda}_{U,V}\colon \rI(U) \otimes_{A_H} \rI(V) \to \rI(U \otimes V)$ in $\locmod_{\Vect_G^\omega}(A_H)$.

{\samepage The morphism $\bar{\Lambda}_{U,V}$ is inverse to $\ov{\nu}_{U,V}$. Indeed, we compute that
\begin{gather*}
 \bar{\nu}_{U,V}\bar{\Lambda}_{U,V}((g\otimes u_d) \otimes (l \otimes v_f)) = \lambda(g,d,l,f)\bar{\nu}_{U,V}\big(g \otimes (u_d \otimes g^{-1}l\cdot v_f)\big) \\
 \phantom{\bar{\nu}_{U,V}\bar{\Lambda}_{U,V}((g\otimes u_d) \otimes (l \otimes v_f)) }{}=\gamma(g)\big(d,g^{-1}lfl^{-1}g\big)\lambda(g,d,l,f)\big((g \otimes u_d) \otimes \big(g \otimes g^{-1}l\cdot v_f\big)\big)\\
\phantom{\bar{\nu}_{U,V}\bar{\Lambda}_{U,V}((g\otimes u_d) \otimes (l \otimes v_f)) }{} =\tau\big(g,g^{-1}l\big)(f)\gamma(g)\big(d,g^{-1}lfl^{-1}g\big)\lambda(g,d,l,f)\\
\phantom{\bar{\nu}_{U,V}\bar{\Lambda}_{U,V}((g\otimes u_d) \otimes (l \otimes v_f))= }{} \cdot ((g \otimes u_d) \otimes (l \otimes v_f))\\
\phantom{\bar{\nu}_{U,V}\bar{\Lambda}_{U,V}((g\otimes u_d) \otimes (l \otimes v_f)) }{} =((g \otimes u_d) \otimes (l \otimes v_f)),\\
 \bar{\Lambda}_{U,V}\bar{\nu}_{U,V}(g \otimes (u_d \otimes v_f)) =\gamma(g)(d,f) \bar{\Lambda}_{U,V}((g \otimes u_d) \otimes (g \otimes v_f)) \\
\phantom{\bar{\Lambda}_{U,V}\bar{\nu}_{U,V}(g \otimes (u_d \otimes v_f))}{} = \lambda(g,d,g,f)\gamma(g)(d,f)g \otimes (u_d \otimes v_f)\\
\phantom{\bar{\Lambda}_{U,V}\bar{\nu}_{U,V}(g \otimes (u_d \otimes v_f))}{} = g \otimes (u_d \otimes v_f).
\end{gather*}
Hence $\bar{\nu}$ is a natural isomorphism and thus $I$ is a monoidal functor.}

To see that $I$ is a braided monoidal functor, consider the diagram
\[
\vcenter{\hbox{\xymatrix{
\rI(U) \otimes \rI(V)\ar[r]\ar[d]^{c_{\rI(U),\rI(V)}} & \rI(U) \otimes_{A_H} \rI(V) \ar[r]^-{\bar{\nu}_{U,V}}\ar[d]^{c'_{U,V}} & \rI(U\otimes V) \ar[d]^{\rI(c_{U,V})}\\
\rI(V) \otimes \rI(U)\ar[r] & \rI(V) \otimes_{A_H} \rI(U) \ar[r]^-{\bar{\nu}_{V,U}} & \rI(V\otimes U).
}}}
\]
The left-most square commutes by definition of the braiding in $\locmod_{\cZ(\Vect_G^\omega)}(A_H)$, and the perimeter commutes by naturality. Hence the right-most square commutes, which is exactly the condition for the functor $\rI$ to be compatible with the braiding.

Now assume $|G:H|\in \Bbbk^\times$. Then as, by Lemma~\ref{lem:AH-rigidFrob}, $A_H$ is a rigid Frobenius algebra, $\locmod_{\cZ(\Vect_G^\omega)}(A_H)$ is a ribbon category by \cite{KO,LW3}. To check compatibility with the twist, recall that $I$ is a ribbon functor to $\cZ\big(\Vect_G^\omega\big)$, see Proposition~\ref{prop:braidedFrob}. Further recall the explicit form of the ribbon twist $\hat{\theta}_V$ on categories of local modules over $A=A_H$ from \cite[Proposition~4.23]{LW3},
\[\hat{\theta}_V= a^r_V(\theta_V\otimes \ide_A)(a_V^r\otimes \ide_A)(\ide_V\otimes q),\]
where $d=\dim_j(A)$ and $q\colon \one \to A\otimes A$ is an inverse to the pairing $p=\varepsilon_A m$. In the case of $A=A_H$,
\[p(\delta_{gH}\otimes \delta_{kH})=\begin{cases}1 & \text{if $g^{-1}k\in H$},\\
0 & \text{else},\end{cases} \qquad q=\sum_i \delta_{g_iH}\otimes \delta_{g_iH}.\]
Thus, we can evaluate the twist $I(V)$, to obtain
\begin{align*}
\hat{\theta}_V(g_j\otimes v_d)&=\sum_{i}\theta_{I(V)}\bigg(\sum_{i} (g_j\otimes v_d)\cdot \delta_{g_iH}\bigg)\cdot \delta_{g_iH}\\
&=(g_j\otimes d\cdot v_d)\cdot \delta_{g_jH}=g_j\otimes d\cdot v_d=I(\theta_V)(g_j\otimes v_d),
\end{align*}
where we use the right $A_H$-action on $I(V)$ from Lemma~\ref{lem:AH-action}, the twist from Remark~\ref{rem:ribbon}, and Proposition~\ref{prop:braidedFrob} in the final step. This proves that $I$ is a ribbon functor to the category of local modules over $A_H$ as claimed.
\end{proof}

The following result was proved in \cite[Theorem 3.7]{DS}. We give a proof using the functor $\rI$ defined above.

\begin{Theorem}\label{thm:Iequiv}
The functor $\rI$ defines an equivalence of monoidal categories between $\cZ\big(\Vect_H^\omega\big)$ and $\locmod_{\cZ(\Vect_G^\omega)}(A_H)$.
\end{Theorem}
\begin{proof}
We first show that $I$ is faithful. On morphisms, $I$ is given by the map
\[\Hom_{\cZ(\Vect_H^\omega)}(V,W)\to \Hom_{\locmod_{\cZ(\Vect_G^\omega)}(A_H)}(\Bbbk G \otimes V,\Bbbk G \otimes W), \qquad p\mapsto I(p)=\ide_{A_H}\otimes p.\]
As this map is injective, seen, for example, by restricting to the subspace $1\otimes V$ of $I(V)$, we see that $I$ is faithful.

To prove that $I$ is full, suppose $q\colon \Bbbk G \otimes_{A_H} V\to \Bbbk G \otimes_{A_H} W$ is a morphism in $\locmod_{\Vect_G^\omega}(A_H)$.
Then,
for all $g,k\in G$, $v\in V$, we have
 \[q(g\otimes v) \cdot \delta_{kH} = q((g\otimes v) \cdot \delta_{kH}) =
\begin{cases}
q(g\otimes v) & \text{if $k^{-1}g \in H$,}\\
0 & \text{otherwise}.
\end{cases}
\]
We choose a set of coset representatives $\{g_i\}$ for $H$ in $G$ such that $g_1=1$.
Using the above and equation~\eqref{eq:relationsIV}, we see that the $q(g_i\otimes v)$ is contained in the subspace $g_i\otimes W$ of $I(W)$. Hence, there exists a vector $w_i\in W$ such that
$q(g_i\otimes v) = g_i\otimes w_i$. However, $q$ commutes with the twisted left $G$-action, which implies that
\[
g_i\otimes w_i=q(g_i\otimes v)= q(g_i\cdot (1\otimes v))=g_i\cdot q(1\otimes v)=g_i\cdot (1\otimes w_1)=g_i\otimes w_1.
\]
Thus, $w_i=w_1$ for all $i$. Hence, we obtain a $\Bbbk$-linear map $q'\colon V\to W$, $v\mapsto w_1$. This map satisfies $I(q')=q$, i.e.,
\[q(g\otimes v)=g\otimes q'(v)\in I(W) \qquad \forall g\in G,\quad v\in V.\]
By restricting to $g\in H$, it follows that $q'\colon V\to W$ is a morphism of twisted YD modules over~$H$. This proves that $I$ is full.

It remains to show that $I$ is essentially surjective. For this, take a local module $L \in \locmod_{\cZ(\Vect_G^\omega)}(A_H)$. The actions of the idempotent elements $\delta_{gH}$ of $A_H$ define a family of idempotent endomorphisms
\[e_g\colon\ L\to L, \qquad l\mapsto l\cdot \delta_{gH},\qquad e_g\in \End_{\cZ(\Vect_G^\omega)}(L).\]
Setting $L^i = \Img(e_{g_i})$ gives a direct sum decomposition $L=\bigoplus\limits_{i}L^{i}$. Here, we use that \[1_{A_H}=\sum\limits_{i}\delta_{g_iH}.\]

We now observe that $l\in L^i$ if and only if~$g_jg_i^{-1}\cdot l\in L^j$. In particular, $l\in L^1$ if and only if~$g_il \in L^i$. Further, $L^1$ is a submodule of $L$ under the left twisted $H$-action. The assumption that $L$ is a local module implies that if $l_d\in L^1$ has degree $|l|=d$, then
$l=l\cdot \delta_{dH}$. We can write~$d=g_i h$, with $h\in H$, and find that $l\in L^i$. However, as the subspaces $L^i$ intersect trivially it follows that $d\in H$. Thus, $L^1$ correspond to an object in $\cZ\big(\Vect_H^\omega\big)$.

Using the twisted right $G$-action on $L$, we define a map
\begin{align*}
\pi\colon\ \Bbbk G \otimes L_1 \longrightarrow L, \qquad
g \otimes l \mapsto g\cdot l.
\end{align*}
The map $\pi$ is surjective. Indeed, $L_1$ is given by all elements of the form $l\cdot \delta_{H}$, with $l\in L$. Now,
\begin{align*}
g_i \left(\frac{\big(g_i^{-1}l\big)}{\gamma(g_i)(g_i,|l|)}\cdot \delta_{H}\right)&=\frac{\tau(g_i)\big(g_i|l|g_i^{-1},1\big)}{\gamma(g_i)(g_i,|l|)}\big(g_i\big(g_i^{-1}l\big)\big)\cdot (g_i\delta_{H})\\
&=\frac{\gamma(g_i)(g_i,|l|)}{\gamma(g_i)(g_i,|l|)}l\cdot \delta_{g_iH}=l\cdot \delta_{g_iH}.
\end{align*}
Thus, $l\cdot \delta_{g_iH}$ is in the image of $\pi$ and hence, for any $l\in L$,
$l=\sum_i l\cdot \delta_{g_iH}\in \Img(\pi).$
It follows that
\[\pi(gh \otimes v_d)=(gh)l=\tau(g,h)(d)^{-1}g(hl)= \tau(g,h)(d)^{-1}g \otimes h v_d.\]
Thus, by equation~\eqref{eq:relationsIV}, $\pi$ descents to a quotient map
$\ov{\pi}\colon I\big(L^1\big) \to L$
which is still surjective. The right twisted action by $g\in G$ gives an isomorphism of vector spaces and hence $\dim(L^i)=\dim\big(L^1\big)$ for all $i$. This shows that
\[\dim L = |G:H|\dim_\Bbbk L^1=\dim_\Bbbk I\big(L^1\big).\]
Thus, $\ov{\pi}$ is injective and hence gives an isomorphism $I\big(L^1\big)\cong L$.
\end{proof}

\begin{Corollary}
If $|G:H|\in \Bbbk^\times$, then the equivalence from Theorem~{\rm \ref{thm:Iequiv}} is an equivalence of ribbon categories.
\end{Corollary}
\begin{proof}
This statement is now a direct consequence of Theorem~\ref{thm:Iequiv} and Lemma~\ref{lem:Imonoidal}.
\end{proof}

\begin{Example}
Let $H=\{1\}$ be the trivial subgroup of $G$. Then it follows that \[\locmod_{\cZ(\Vect_G^\omega)}(A_{\{1\}})\simeq \Vect_{\Bbbk},\]
i.e., $A_{\{1\}}$ is a trivializing algebra in $\cZ\big(\Vect_G^\omega\big)$ provided that $|G|\in \Bbbk^\times$. In fact, for any subgroup~$H$ of $G$ with $|G:H|\in \Bbbk^\times$, we obtain $\cZ\big(\Vect_H^\omega\big)$ as local modules over an algebra in $\cZ\big(\Vect_G^\omega\big)$. This provides a correspondence between ribbon categories, cf.\ \cite[Section~1.4]{FFRS}.
\end{Example}

\subsection{The classification of rigid Frobenius algebras}

In this section, we apply the Frobenius monoidal functors from Section~\ref{sec:I} in order to recover the classification of rigid Frobenius algebras in $\cZ(\lmod{\Bbbk G})\simeq \cZ(\Vect_G)$ from \cite[Theorem 3.15]{DS}, and generalize this result to algebraically closed fields of arbitrary characteristic. For the case of a trivial $3$-cocycle $\omega=1$, this recovers \cite[Theorem~3.5.1]{Dav3}, for $\cha \Bbbk=0$, and \cite[Theorem~6.14]{LW3} for general algebraically closed fields. In fact, as in \cite{DS}, we obtain a classification of all connected \'etale algebras in $\cZ\big(\Vect_G^\omega\big)$ which turn out to have trivial twist and are, hence, rigid Frobenius algebras if their quantum dimension is non-zero.

\begin{Notation}[input data $H$, $N$, $\omega$, $\kappa$, $\epsilon$] \label{not:cocycledata}
Let $H$ be a finite group with a $3$-cocycle $\omega\in H^3(H,\Bbbk^\times)$, $N\triangleleft H$ a normal subgroup.
Further, let $\kappa \colon N\times N\to \Bbbk^\times$ satisfy
\begin{align}\label{eq:tau}
 \omega(n,m,k) =\kappa(n,m)\kappa(m,k)^{-1}\kappa(nm,k)\kappa(n,mk)^{-1}, \qquad \kappa(n,1)=\kappa(1,n)=1,
\end{align}
for all $n,m,k\in N$. In addition, let
\[\epsilon\colon\ H\times N\to \Bbbk^\times, \qquad (h,n)\mapsto \epsilon_h(n)\]
be a map satisfying, for all $h,k\in H$ and $n,m\in N$, that
\begin{gather}
 \tau(h,k)(n)=\frac{\epsilon_h\big(knk^{-1}\big)\epsilon_k(n)}{\epsilon_{hk}(n)}, \label{eq:Davcond1}\\
 \gamma(h)(n,m)=\frac{\epsilon_h(nm)}{\epsilon_h(n)\epsilon_h(m)}\cdot\frac{\kappa\big(hnh^{-1},hmh^{-1}\big)}{\kappa(n,m)}, \label{eq:Davcond2}\\
 \kappa\big(nmn^{-1},n\big)=\epsilon_n(m)\kappa(n,m).
 \label{eq:Davcond3}
\end{gather}
In particular, the normalized condition on $\kappa$, along with \eqref{eq:Davcond2} and \eqref{eq:Davcond3}, respectively, imply that
\begin{align*}
 \epsilon_{h}(1) = 1, \qquad \text{and} \qquad \epsilon_{1}(n) = 1. 
\end{align*}
\end{Notation}

\begin{Remark}\label{rem:Ntriv}
 The maps $\epsilon(h,n):=\epsilon_{h}(n)$ and $\kappa$ define a normalized element $\epsilon\oplus \kappa$ in the truncated total complex $\wF^2(H,N,\Bbbk^\times)$, where $H$ acts on $N$ by conjugation, see Appendix~\ref{appendix:coh-crossed}. Equations \eqref{eq:tau}--\eqref{eq:Davcond2} are equivalent to
 \begin{equation*}
 \deri_{\Tot}^2(\epsilon\oplus \kappa)=\tau\oplus \gamma \oplus \omega,
 \end{equation*}
 where $\tau(h_1,h_2,n)=\tau(h_1,h_2)(n)$, $\gamma(h,n_1,n_2)=\gamma(h)(n_1,n_2)$.
 In particular, the restriction of $\omega$ to a $3$-cocycle on $N$ is trivial in $H^3(N,\Bbbk^\times)$.
\end{Remark}

Now assume given a finite group $G$ with a $3$-cocycle $\omega\in H^3(G,\Bbbk^\times)$.
The isomorphism classes of rigid Frobenius algebras in $\cZ\big(\Vect_G^\omega\big)$ will be parametrized by the data in Notation~\ref{not:cocycledata} for $H$ a subgroup of $G$ and the restriction of $\omega$ to $H$. We will denote such algebras by $A=A(H,N,\kappa,\epsilon)$ in Theorem~\ref{thm:classification} below. These algebras $A$ will be equal to $\rI(B)$ for a rigid Frobenius algebra $B=B(N,\kappa,\epsilon)$ in $\cZ\big(\Vect_{H}^\omega\big)$ using the Frobenius monoidal functor $\rI$ from Propositions~\ref{prop:Frob} and~\ref{prop:braidedFrob}.

\begin{Proposition}[algebras $B(N,\kappa,\epsilon)$] \label{prop:B-algebra} Assume given a tuple $(H,N,\omega,\kappa,\varepsilon)$ as in Notation~{\rm \ref{not:cocycledata}}.
Consider the $\Bbbk$-vector space $B(N,\kappa, \epsilon)$ with $\Bbbk$-basis $\{ e_n \mid n\in N\}$, and define
 \begin{enumerate}\itemsep=0pt
 \item[$(i)$] $h \cdot e_n=\epsilon_h(n) e_{hnh^{-1}}$, for $h\in H$;
 \item[$(ii)$] $\delta(e_n)=n\otimes e_n$, that is, $e_n$ is homogeneous of degree $n\in H$;
 \item[$(iii)$] multiplication $m_B$ given by $e_n e_m = \kappa(n,m)^{-1} e_{nm}$ for all $ n,m\in N$;
 \item[$(iv)$] unit $1_B = e_{1}$.
 \end{enumerate}
 Then $B(N,\kappa, \epsilon)$ is a connected, commutative algebra in $\cZ\big(\Vect_H^\omega\big)$ described as a twisted YD module.
\end{Proposition}

The following is an analogue of \cite[Proposition~3.11]{DS}, \cite[Proposition~3.4.2]{Dav3}.
\begin{Proposition}\label{prop:B-classification}
Recall the data from Notation~{\rm \ref{not:cocycledata}}.
 Let $B$ be an \'etale algebra in $\cZ\big(\Vect_H^\omega\big)$ such that $B_1=\Bbbk$ and $\dim_j(B)\neq 0$.
 Then $B$ is isomorphic as an algebra in $\cZ\big(\Vect_H^\omega\big)$ to~$B(N,\kappa,\epsilon)$ for $N=\operatorname{Supp}(B)=\Set{h\in H \mid B_h\neq 0 }$.
\end{Proposition}

\begin{proof}
As $B$ is \'etale, it is commutative and separable by definition. Separability implies that the restriction of the multiplication defines a non-degenerate pairing $B_h\otimes B_{h^{-1}}\to B_1=\Bbbk$, e.g., by Proposition~\ref{prop:rigidFrobcharacterization}.
This implies that any non-zero element $b\in B_h$ is a unit in $B$. Thus, $ab\neq 0$ in $B_{hk}$ provided that $a\in B_h$, $b\in B_k$ are non-zero. This shows that $N=\operatorname{Supp}(B)$ is a subgroup of $H$.
Further, $N$ is a normal subgroup of $H$ by the twisted YD condition.
One argues as in \cite[Lemma~3.4.1]{Dav3} that $\dim_\Bbbk B_h\leq 1$ for any $h\in H$.

Now, we can choose a $\Bbbk$-basis $\Set{e_n}_{n\in N}$ for $B$. Then, as $\dim_\Bbbk B_h\leq 1$, the multiplication in $B$ is determined by scalars $\kappa(n,m)\in \Bbbk^\times$ satisfying
\begin{equation*}e_n e_m = \kappa(n,m)^{-1}e_{nm} \qquad \forall n,m\in N.
\end{equation*}
Further, the left $\Bbbk H$-action is determined by scalars $\epsilon_h(n)\in \Bbbk^\times$ which satisfy
\begin{equation*}
 h\cdot e_n = \epsilon_h(n)e_{hnh^{-1}} \qquad \forall h\in H, \quad n\in N.
\end{equation*}
Together with the given $3$-cocycle $\omega$ this gives us a tuple $(H,N,\omega,\kappa,\varepsilon)$ as in Notation~\ref{not:cocycledata}, where it follows from $B$ being an algebra in $\cZ\big(\Vect_H^\omega\big)$ that the conditions in equations \eqref{eq:tau}--\eqref{eq:Davcond3} hold. In particular,
\begin{itemize}\itemsep=0pt
 \item equation~\eqref{eq:tau} corresponds to $m_B$ being associative and unital,
 \item equation~\eqref{eq:Davcond1} corresponds to $B$ being a YD module,
 \item equation~\eqref{eq:Davcond2} corresponds to $m_B$ being a morphism of YD modules,
 \item equation~\eqref{eq:Davcond3} corresponds to $m_B$ being commutative in $\cZ\big(\Vect_H^\omega\big)$.
\end{itemize}
This completes the proof.
\end{proof}

\begin{Proposition}\label{prop:Brigid}
Assume that $|N| \in \Bbbk^\times$.
The algebras $B=B(N,\kappa, \epsilon)$ defined in Proposition~{\rm \ref{prop:B-algebra}} are rigid Frobenius algebras in $\cZ\big(\Vect_{H}^\omega\big)$ with coalgebra structure given by
 \[\Delta_B(e_n) = \sum_{m\in N}\kappa\big(m,m^{-1}n\big) \;e_{m}\otimes e_{m^{-1}n}, \qquad \varepsilon_B(e_n) = \delta_{n,1}, \qquad \text{for all}\quad n\in N.\]
\end{Proposition}
\begin{proof}

Consider the algebra $B=B(N,\kappa, \epsilon)$. One readily verifies that the conditions in Notation~\ref{not:cocycledata} are sufficient to ensure that $B(N,\kappa,\varepsilon)$ is a commutative algebra in $\cZ\big(\Vect_H^\omega\big)$ (cf.\ the bullet points in the proof of Proposition~\ref{prop:B-algebra}).

It remains to check that such an algebra $B$ in, in fact, a rigid Frobenius algebra. This is argued as in \cite[Proposition~6.12\,(2)]{LW3}. First, $B$ is connected since \[\Hom_{\cZ\big(\Vect_H^\omega\big)}(\one,B)\subseteq \Hom_{\lcomod{\Bbbk H}}(\one,B)\subseteq B_1,\]
with the containing space being $1$-dimensional.

Next, we define $\varepsilon_B$ and check that it is indeed a morphism of twisted $H$-YD modules. Further, one checks, using the map
\[q=\sum_{n\in N}\frac{\kappa\big(n^{-1},n\big)}{\omega\big(n^{-1},n,n^{-1}\big)}e_n\otimes e_{n^{-1}}\]
that the pairing $p:= \varepsilon_B m_B\colon B\otimes B\to \one$ is non-degenerate.
Then, the coproduct
\[\Delta_B(e_n) = \sum_{m\in N}\frac{\kappa\big(m^{-1},m\big)}{\kappa\big(m^{-1},n\big)}\frac{\omega^{-1}\big(m,m^{-1},n\big)}{\omega\big(m^{-1},m,m^{-1}\big)}\;e_{m}\otimes e_{m^{-1}n} \stackrel{\text{\eqref{eq:tau}}}{=}\sum_{m\in N}\kappa\big(m,m^{-1}n\big) \;e_{m}\otimes e_{m^{-1}n}\]
is obtained from the multiplication and the pairing
$\varepsilon$ following Remark~\ref{rem:Frobpairing}.
This way, $\Delta_B$, $\varepsilon_B$ make $B$ a coalgebra in $\cZ\big(\Vect_H^\omega\big)$.

In fact, the algebra and coalgebra structures satisfy the Frobenius conditions from Definition~\ref{def:co-alg-Frob}\,(3). Verifying that $B$ is a special Frobenius algebra amounts to the computations that~$m_B \Delta_B (e_n)=|N|\ide_B$, and $\varepsilon_B(1_B)=1$.
Hence, $B$ is a rigid Frobenius algebra since~$|N|\neq 0$.
\end{proof}

\begin{Remark}
By forgetting the twisted YD module structure, we can view $B(N,\kappa,\epsilon)$ as a~special Frobenius algebra in $\Vect_G^{\omega}$. The images under the forgetful functor are twisted group algebras $A(N,\psi)$ \cite{Nat,Ost}. These twisted group algebras were used to classify indecomposable module categories over $\Vect_G^\omega$ by Ostrik and Natale, cf.\ \cite[Example~9.7.2]{EGNO}. An explicit Frobenius algebra structure for $A(N,\psi)$ was given in \cite[Proposition 5.7]{WINART}.
Under the identifications $n=g$, $m=gh$ and $\kappa = \psi^{-1}$, the multiplication, unit and counit of $B(N,\kappa,\epsilon)$ match those of $A(N,\psi)$ up to normalisation.
The coproduct formula for $B(N,\kappa,\epsilon)$ becomes
\[
\Delta_B(e_g) = \sum_{h\in N} \kappa\big(gh,h^{-1}\big) \;e_{gh}\otimes e_{h^{-1}}.
\]
Compared to that of $A(N,\psi)$, the coproduct is the same, up to normalisation by $1/|N|$.
 The above results hence show that $\epsilon\colon H\times N \to \Bbbk^\times $ parametrize lifts of the algebras $A(N,\kappa)$ to rigid Frobenius algebras $B(N,\kappa,\epsilon)$ in the center. Such a lift can only exist if $\epsilon$, $\kappa$ satisfy the conditions of Notation~\ref{not:cocycledata}.
\end{Remark}

We will now give an explicit description of the commutative Frobenius algebras $A=\rI(B)$ in~$\cZ\big(\Vect_G^\omega\big)$ for $B=B(N,\kappa,\epsilon)$ as above. Note that Corollary~\ref{cor:Ipreserves} ensures that $A$ is a~commutative Frobenius algebra, but we will see that it is, in fact, also a rigid Frobenius algebra.

\begin{Definition}[$A(H,N,\kappa,\epsilon)$]\label{def:A}
Let $G$ be a group with $\omega\in C^3(G,\Bbbk^\times)$, a subgroup $H$ of $G$,
and a tuple $(N,\kappa,\varepsilon)$ as in Notation~\ref{not:cocycledata}.
We define $A=A(H,N,\kappa,\epsilon)$ to be the commutative Frobenius algebra $\rI(B)$ for $B=B(N,\kappa,\varepsilon)$ from Proposition~\ref{prop:Brigid}.

Further, we fix a coset decomposition $G=\bigsqcup_{i\in I} g_i H$.
\end{Definition}

\begin{Lemma}\label{lem:AFrobenius}
 Explicitly, we can describe the structure of $A=A(H,N,\kappa,\epsilon)$ as a Frobenius algebra in $\cZ\big(\Vect_G^\omega\big)$ as follows.
\begin{enumerate}\itemsep=0pt
 \item[$(a)$] $A$ is the quotient $\Bbbk$-vector space spanned by $\{ a_{g,n} \mid g \in G, n\in N\}$, subject to the relations
 \begin{align}
 a_{gh,n}=\tau(g,h)(n)^{-1}\epsilon_h(n)a_{g,hnh^{-1}}, \qquad \forall h\in H.\label{eq:Davquotrel}
\end{align}
\item[$(b)$] The twisted YD module structure is given by
 \begin{enumerate}\itemsep=0pt
 \item[$(i)$] left $\Bbbk G$-coaction given by $\delta(a_{g,n})=g ng^{-1}\otimes a_{g,n}$;
 \item[$(ii)$] twisted $G$-action given by $k\cdot a_{g,n}=\tau(k,g)(n)a_{kg,n}$ for $k\in G$.
 \end{enumerate}
\item[$(c)$] The Frobenius algebras structure is given by the
 \begin{enumerate}
 \item[$(iii)$] multiplication $m_A$ given by \[a_{g,n}a_{g,m}=\gamma(g)(n,m)^{-1}\kappa(n,m)^{-1}a_{g,nm},
 \] for $g\in G$ and $n,m\in N$, and $a_{g,n}a_{k,m}=0$ if $kH\neq gH$;
 \item[$(iv)$] unit $u_A$ given by $1_A =\sum\limits_{i\in I}a_{g_i,1}$;
 \item[$(v)$] coproduct $\Delta_A$ given by
 \[\Delta_A(a_{g,n}) = \sum_{m\in N}\gamma(g)\big(m,m^{-1}n\big)\kappa\big(m,m^{-1}n\big) \;a_{g,m}\otimes a_{g,m^{-1}n},\]
 for all $g\in G$ and $n\in N$;
 \item[$(vi)$] counit $\varepsilon_A$ given by
 $\varepsilon_A(a_{g,n}) = \delta_{n,1}.$
 \end{enumerate}
\end{enumerate}
\end{Lemma}
\begin{proof}We set $a_{g,n}:=g\otimes e_n\in A=\rI(B)$.
The relations on $A$ in \eqref{eq:Davquotrel} are then derived from equation~\eqref{eq:relationsIV} using the twisted $H$-action from Proposition~\ref{prop:B-algebra}\,(i). This proves (a). To obtain the formulas in (b) we apply the twisted YD module structure on $A=\rI(B)$ from Lemma~\ref{lem:IV-twistedYD}.

To find the Frobenius algebra structure on $A$ displayed in (c) we use Corollary~\ref{cor:Ipreserves}. Thus, computing multiplication and unit involves the lax monoidal structure of $\rI$, see Lemma~\ref{lem:lax}. It suffices to evaluate the product on elements $a_{g,n}a_{g,m}$ as a product $a_{g,n}a_{k,n}=0$ if $kH\neq gH$. If~$kH=gH$, we can find $n'$ such that $a_{k,n}=a_{g,n'}$ as argued before Lemma~\ref{lem:lax}.
We compute
\[a_{g,n}a_{g,m}=\gamma(g)(n,m)^{-1}(g\otimes e_n\cdot e_m)=\gamma(g)(n,m)^{-1}\kappa(n,m)^{-1}a_{g,nm}.\]
The unit is given by $\Bbbk \to A, 1\mapsto 1_A= \sum_{i}g_i\otimes 1=\sum_ia_{g_i,1}$, using equation~\eqref{eq:lax-unit}. Finally, the coproduct and unit are computed using the oplax monoidal structure on $\rI(A)$ from Lemma~\ref{lem:oplax}, i.e., $\Delta_A=\nu_{B,B}\rI(\Delta_B)$. This gives the claimed formulas.
\end{proof}

We obtain the following theorem generalizing \cite[Theorem 3.15]{DS} to arbitrary characteristic.

\begin{Theorem} \label{thm:classification} Let $G$ be a finite group with $\omega\in C^3(G,\Bbbk^\times)$, a subgroup $H$ of $G$ and a tuple $(N,\kappa,\varepsilon)$ as in Notation~{\rm \ref{not:cocycledata}}.
\begin{enumerate}\itemsep=0pt
 \item[$(a)$] If $|N|\cdot|G:H| \in \Bbbk^{\times}$, then the algebra $A=A(H,N,\kappa, \epsilon)$ from Definition~{\rm \ref{def:A}} is a rigid Frobenius algebra in $\cZ\big(\Vect_G^\omega\big)$ of dimension $\dim_j(A)=|N| |G:H|$.
 \item[$(b)$] Every connected \'etale algebra in $\cZ\big(\Vect_G^\omega\big)$ is of the form $A(H,N,\kappa, \epsilon)$ for some choice of data $H$, $N$, $\gamma$, $\epsilon$ and has trivial twist $\theta_A$.
 \item[$(c)$] Every rigid Frobenius algebra in $\cZ\big(\Vect_G^\omega\big)$ is isomorphic to one of the form $A(H,N,\kappa, \epsilon)$ for some choice of data $H$, $N$, $\gamma$, $\epsilon$ as above, with $|N| \cdot |G:H| \in \Bbbk^{\times}$.
\end{enumerate}
\end{Theorem}
\begin{proof}
To prove part (a), observe that $A=A(H,N,\kappa,\epsilon)$ is a commutative Frobenius algebra in $\cZ\big(\Vect_G^\omega\big)$ by Corollary~\ref{cor:Ipreserves} and Proposition~\ref{prop:Brigid}. It remains to check that $A$ is connected and special. First, $A$ is connected since
$\Hom_{\cZ(\Vect_G^\omega)}(\one,A)= (A_1)^G$, the space of $G$-invariant elements in the $\Bbbk G$-module $A_1$. This space is $1$-dimensional as $A_1=\Bbbk(G/H)$ and $G$ acts by left translation. Next, to see that $A$ is special we compute
$m_A\Delta_A(a_{g,n})=|N| a_{g,n}$ as in Proposition~\ref{prop:Brigid}. Further, $\varepsilon_A(1_A)=|G:H|\ide_{\one}$,
and by assumption, both scalars $|N|$ and $|G:H|$ are non-zero. Thus, $A$ is a rigid Frobenius algebra of the claimed quantum dimension.

To prove part (b), assume $A$ is a connected \'etale algebra in $\cZ\big(\Vect_G^\omega\big)$. Following the strategy from \cite[Corollary~3.3.5]{Dav3}, we consider the subalgebra $A_1$. We note that $A_1$ is a subobject of $A$ in $\cZ\big(\Vect_G^\omega\big)$. The twisted $\Bbbk G$-action on $A$ restricts to an (untwisted) $\Bbbk G$-module structure on $A_1$.
Thus, $A_1$ corresponds to an algebra in the symmetric monoidal category of $\Bbbk G$-modules. As the braiding of $A$ restricted to $A_1\otimes A_1$ is symmetric, $A_1$ is, in fact, a commutative algebra over $\Bbbk$.
Now, $A$ is separable by Proposition~\ref{prop:rigidFrobcharacterization} and this implies that $A_1$ is also separable. Indeed, $A_1$ is also a connected commutative algebra in $\cZ\big(\Vect_G^\omega\big)$ and as such it is separable if and only if the pairing $\varepsilon_\sharp m$ is non-degenerate, cf.\ \cite[Section 2.2]{Dav3} and \cite[Section 3.3]{LW3}. But non-degeneracy of this pairing on $A$ implies non-degeneracy of the restriction to $A_1$.
Hence, $A_1$ is a connected \'etale algebra in $\lmod{\Bbbk G}$ given that its $G$-grading is trivial. Viewing $A$ as a $\Bbbk$-algebra, it follows that~$A_1\cong \Bbbk ^n$ for some $n$ since $A$ is algebraically closed. The primitive central idempotents of $A_1$ are a $G$-set by restricting the $\Bbbk G$-action on $A_1$. Thus, using indecomposability of $A$, $A_1\cong \Bbbk(G/H)$ is of the form $A_H$ from Example~\ref{ex:AH} for some subgroup $H\leq G$. This argument appears in \cite[Theorem~2.2]{KO} in the semisimple case.

The multiplication of $A$ restricts to a right action of $A_1$ on $A$, which makes $A$ a local module over $A_1$ using commutativity of $A$ in $\cZ\big(\Vect_G^\omega\big)$. Thus, by the equivalence in Theorem~\ref{thm:Iequiv}, $A\cong \rI(B)$ for a connected \'etale algebra $B$ in $\cZ\big(\Vect_H^\omega\big)$.
Now, $\dim_\Bbbk B_1=1$ as \[\dim_\Bbbk A_1\geq (\dim_\Bbbk B_1)(\dim_\Bbbk A_H)=(\dim_\Bbbk B_1)(\dim_\Bbbk A_1).\]
Thus, $B$ is isomorphic, as an algebra in $\cZ\big(\Vect_H^\omega\big)$ to an algebra of the form $B(N,\kappa, \epsilon)$ by Proposition~\ref{prop:B-classification}.

To prove part (c), assume that $A$ is any rigid Frobenius algebra in $\cZ\big(\Vect_G^\omega\big)$, then $A$ is connected \'etale by Proposition~\ref{prop:rigidFrobcharacterization}. Thus, as in part (b), $A\cong A(H,N,\kappa, \varepsilon)$ for some choice of data as in Notation~\ref{not:cocycledata}.
We compute that
\begin{align*}
& m_A\Delta_A(a_{g,n})=\sum_{m\in N}a_{g,n}=|N| a_{g,n},\qquad
 \varepsilon_A 1_A = |G : N|.
\end{align*}
Hence,
\[\dim_j(B)=\evr_B \coev_B=\varepsilon_B\Delta_Bm_B(1_B)=|N|\cdot |G : N|\]
computes the \emph{quantum dimension} of $B$ in $\cZ\big(\Vect_H^\omega\big)$, cf.\ \cite[equation (3.7)]{LW3}, independently of choice of a pivotal structure for $\cZ\big(\Vect_H^\omega\big)$. Now, $A$ is rigid Frobenius if $|N|\cdot |G : N|\neq 0$.
\end{proof}

\begin{Remark}\label{rem:ribbon}
Note that the classification results of this section do not depend on the choice of a particular ribbon structure on $\cZ\big(\Vect_G^\omega\big)$ and different choices of ribbon structures may be used in Theorem~\ref{thm:classification}. By default, we use the ribbon structure on $\cZ(\Vect_{G}^\omega)$ detailed in Proposition~\ref{prop:Zribbon}.
\end{Remark}

\begin{Corollary}\label{cor:modular}
Let $A:=A(H,N,\kappa,\epsilon)$ be an algebra in $\cZ\big(\Vect_G^\omega\big)$ as defined in Definition~{\rm \ref{def:A}} and assume $|N|\cdot |G:H|\in \Bbbk^\times$.
Then the category $\locmod_{\cZ(\Vect_{G}^\omega)}(A)$ is a modular category.
\end{Corollary}
\begin{proof}
By Theorem~\ref{thm:classification}, $A$ is a rigid Frobenius algebra. Hence, by \cite[Theorem 4.12]{LW3}, $\locmod_{\cZ(\Vect_{G}^\omega)}(A)$ is a modular category. Given a ribbon structure on $\cZ\big(\Vect_G^\omega\big)$, cf.\ Remark~\ref{rem:ribbon}, $\locmod_{\cZ(\Vect_{G}^\omega)}(A)$ is a ribbon category by \cite[1.17.~Theorem]{KO} or \cite[Proposition~4.18]{LW3}.
\end{proof}

Note that if $\cha \Bbbk$ does not divide $|G|$, then $\cZ\big(\Vect_G^\omega\big)$ is semisimple (see, e.g.,~\cite[Corollary~13.2.3]{Rad}). Hence, in this case $\locmod_{\cZ(\Vect_{G}^\omega)}(A)$ is a modular fusion category.

\begin{Lemma}\label{lem:AH-subalg} Assume given a datum $(H,N,\kappa,\epsilon)$ as in Notation~{\rm \ref{not:cocycledata}}.
\begin{enumerate}\itemsep=0pt
 \item[$(a)$]
The algebra $A_H$ is isomorphic to the subalgebra of $A=A(H,N,\kappa,\epsilon)$ generated by the elements $g\otimes 1$ as an algebra in $\cZ\big(\Vect_G^\omega\big)$.
\item[$(b)$]
The subalgebra generated by the elements $1\otimes e_n$ is isomorphic to $B=B(N,\kappa,\epsilon)$ as an algebra in $\cZ\big(\Vect_H^\omega\big)$.
\end{enumerate}
\end{Lemma}

\begin{Proposition}\label{prop:LocmodA}
Assume $|N|\cdot |G:H|\in \Bbbk^\times$. The induced functor
\[I\colon\ \locmod_{\cZ\big(\Vect_H^\omega\big)}(B)\to \locmod_{\cZ(\Vect_G^\omega)}(A), \qquad V\mapsto I(V)\]
defines an equivalence of ribbon categories.
\end{Proposition}
\begin{proof}
Using Proposition~\ref{prop:braidedFrob} and Corollary~\ref{cor:Ipreserves}, $I$ defines a functor
\[\Rep_{\cZ\big(\Vect_H^\omega\big)}(B)\to \Rep_{\cZ(\Vect_G^\omega)}(A).\]
The right $A=I(B)$-action is defined using the lax monoidal structure of $I$. As $I$ is a braided lax monoidal functor, it preserves local modules.

Note that, as $I$ is fully faithful, we know
\[I\colon\ \Hom_{\cZ\big(\Vect_H^\omega\big)}(V,V')\to \Hom_{\locmod_{\cZ(\Vect_G^\omega)}(A_H)}(I(V),I(V')) \]
is fully faithful.
Thus, we need to show that a morphism $I(\alpha)\colon I(V)\to I(V')$ is a morphism of $A$-modules if and only if $\alpha\colon V\to V'$ is a morphism of $B$-modules. Indeed, if $I(\alpha)$ is a morphism of $A$-modules,
\[I(\alpha)((1\otimes v)\cdot (1\otimes e_n))=1\otimes \alpha(v\cdot e_n)=1\otimes \alpha(v)\cdot e_n.\]
Thus, $\alpha$ is a morphism of $B$-modules. The converse implication is clear.

Now, by Lemma~\ref{lem:AH-subalg}, we see that local $A$-module $W$ is also a local $A_H$-module. Thus, by Theorem~\ref{thm:Iequiv}, $W\cong I(V)$ as an object in $\locmod_{\cZ(\Vect_G^\omega)}(A_H)$ for some object $V$ in $\cZ\big(\Vect_H^\omega\big)$. By construction, $V=W^1$, which is the image of the idempotent element $1\otimes 1\in A$. This idempotent is central and hence defines an $A$-submodule which is also local. In particular, $V$ is a local $B$-module in $\cZ\big(\Vect_H^\omega\big)$.
We need to show that $I(V)\cong W$ as $A$-modules.

By \ref{lem:AH-subalg}, the $A_H$-module structure of both $W$ and $I(V)$ are induced from the respective $A$-module actions, $\rho_W\colon W \otimes A \to W$, $\rho_{I(V)}\colon I(V) \otimes A \to I(V)$.
Thus we have the following commutative diagram
\begin{align*}
 \xymatrix{
 I(V) \otimes A_H \ar@{^{(}->}[r] \ar[d]^-{\pi} & I(V) \otimes A \ar[r]^-{\rho_{I(V)}} \ar[d]^-{\pi} & I(V) \ar[d]^-{\pi} \\
 W \otimes A_H \ar@{^{(}->}[r] & W \otimes A \ar[r]^-{\rho_w} & W. }
\end{align*}
The perimeter commutes as $I(V) \cong W$ as $A_H$-modules, and the left square clearly commutes. As the action is induced by the embedding of the subalgebra $A_H$ into $A$, the right square commutes. Thus $I(V) \cong W$ as $A$-modules and the functor is essentially surjective.

The induced functor $I$ is a monoidal functor, with the monoidal structure being inherited directly.
The functor $I$ is compatible with braidings as braidings of local modules are induced from the braidings of the underlying categories.
Finally, as the functor $I$ is compatible with the ribbon twist, so is the induced functor on local modules, whose ribbon category structure is induced from that of the underlying category, see the proof of Proposition~\ref{prop:braidedFrob}.
\end{proof}

The above proposition shows that, up to equivalence of braided monoidal categories, it suffices to consider the algebra objects $B(N,\kappa,\epsilon)$ in $\cZ\big(\Vect_H^\omega\big)$, i.e., it suffices to consider the case $G=H$. The next proposition addresses when such algebras are isomorphic.

\begin{Proposition} 
Fix $H$ and $\omega\in H^3(H,\Bbbk^\times)$ and let $(N,\kappa,\epsilon)$ and $(N',\kappa',\epsilon')$ be tuples satisfying the conditions of Notation~{\rm \ref{not:cocycledata}}. Then $B=B(N,\kappa,\epsilon)$ and $B'=B(N',\kappa',\epsilon')$ are isomorphic as algebras in $\cZ\big(\Vect_H^\omega\big)$ if and only if $\epsilon'\epsilon^{-1}\oplus \kappa'\kappa^{-1}$ is zero in $\widetilde{H}^2_{\Tot}(H,N,\Bbbk^\times)$.
\end{Proposition}
\begin{proof}
Assume that $\phi\colon H\to H'$ is an isomorphism of algebras in $\cZ\big(\Vect_H^\omega\big)$. Thus, $\dim A=\dim A'$ and hence $|N|=|N'|$.
Then $\phi$ is a morphism of twisted YD modules over $H$ and, in particular, an isomorphism of $G$-graded vector spaces.
This implies that $N=N'$ and $\phi(e_n)=\sigma(n) e_n$ for some scalars $\sigma(n)\in \Bbbk^\times$. Now, as $\phi$ is a morphism of algebras,
\[\phi(e_{n}e_{m})=\phi\big(\kappa(n,m)^{-1}e_{nm}\big)=\kappa(n,m)^{-1}\sigma(nm)e_{nm}=\sigma(n)\sigma(m)\kappa'(n,m)^{-1}e_{nm}.\]
Thus, $\phi$ is a morphism of algebras if and only if
\[\frac{\kappa'(n,m)}{\kappa(n,m)}=\frac{\sigma(n)\sigma(m)}{\sigma(nm)},\]
$\kappa'\kappa^{-1}=\partial^{0,1}(\sigma)$ as claimed.

Further, $\phi$ is a morphism of twisted YD modules. Thus,
\[\phi(h\cdot e_n)=\epsilon_h(n)\sigma\big(hnh^{-1}\big)e_{hnh^{-1}}=\epsilon'_h(n)\sigma(n)e_{hnh^{-1}}=h\cdot \phi(e_n).\]
Thus, $\phi$ is a morphism of twisted YD modules if and only if
\[\frac{\epsilon'_h(n)}{\epsilon_h(n)}=\frac{\sigma\big(hnh^{-1}\big)}{\sigma(n)}.\]
This condition gives that $\epsilon^{-1}\epsilon'=\deri^{0,1}(\sigma)$. Combining, we see that
 \[\big(\epsilon'\oplus \kappa'\big)\cdot (\epsilon\oplus \kappa)^{-1}=\frac{\epsilon'}{\epsilon}\oplus \frac{\kappa'}{\kappa} \]
equals $\deri^{1}_{\Tot}(\sigma)$ and hence is zero in $\widetilde{H}^2_{\Tot}(H,N,\Bbbk^\times)$ as claimed.
\end{proof}

\begin{Remark}\label{rem:simplify}
For $N\triangleleft H$, consider the short exact sequence of groups
\[1\to N\xrightarrow{\iota}H\xrightarrow{\pi} H/N\to 1.\] This induces, via pullback maps, an exact sequence of cochain complexes of abelian groups
\[0\to F^\bullet(H/N,\Bbbk^\times)\xrightarrow{\pi^*} F^\bullet(H,\Bbbk^\times)\xrightarrow{\iota^*} F^\bullet(N,\Bbbk^\times)\to 0,\]
implying that
$H^\bullet(H/N,\Bbbk^\times)=\ker H^\bullet(\iota^*)$. Thus, any $n$-cocycle on $f\colon H^n\to \Bbbk^\times$ such that $ f|_N=0$ in $H^n(N,\Bbbk^\times)$ is equal to a $n$-cocycle $\pi^* \ov{\omega}$, where $\ov{\omega}$ is a normalized $n$-cocycle on the quotient group $H/N$ up to coboundary.
Thus, if $(H,N,\omega,\kappa, \epsilon)$ is a tuple as in Notation~\ref{not:cocycledata}, then by Lemma~\ref{lem:omega-equiv} we can assume, without loss of generality, that $\omega=\pi^*\ov{\omega}$. Then $\tau$, $\gamma$ are also trivial when restricted to inputs from $N$. This follows since
\[\pi^*\ov{\omega}(h_1,h_2,h_3)=\omega(\pi(h_1),\pi(h_2),\pi(h_3))=1\]
as soon as one of the $\pi(h_i)=1\in H/N$, i.e., as soon as one of the $h_i\in N$. However, the definition of $\gamma$, $\tau$ as elements of $\wF^2(H,N,\Bbbk^\times)$ involves at least one input from $N$.
Thus, without loss of generality, isomorphism classes of rigid Frobenius algebras $B(N,\kappa,\epsilon)$ are parametrized by elements $\epsilon\oplus \kappa\in \tilde{H}^2_{\Tot}(H,N,\Bbbk^\times)$.

As a special case, by Corollary~\ref{cor:gen-equiv} we have that $A=R(\one)$ is an algebra in $\cZ(\Vect_G^{\omega})$ such that $\Rep_{\cZ(\Vect_G^{\omega})}(A)\simeq \cZ(\Vect_{H/N}^{\ov{\omega}})$. We observe that
$A=B(N,1,1)$, where $\kappa$ and $\epsilon$ are trivial (i.e., constant functions with value $1\in \Bbbk^\times$). As all conditions in Notation~\ref{not:cocycledata} are trivial for this data, $A$ is a rigid Frobenius algebra in $\cZ(\Vect_G^{\omega})$, cf.\ Lemma~\ref{lem:AH-rigidFrob}.
\end{Remark}

Davydov--Simmons prove the following result on local modules over the Frobenius algebras studied in this section.

\begin{Theorem}[{\cite[Theorem~3.16]{DS}}]\label{thm:DSlocal}
Let $A$ be a rigid Frobenius algebra as in Theorem~{\rm \ref{thm:classification}} and $\pi\colon H\to H/N$ the quotient homomorphism. Then there exists a $3$-cocycle $\ov{\omega}\in C^3(H/N, \Bbbk^\times)$ such that $\pi^*\ov{\omega}=\omega|_{H}$ and an equivalence of ribbon categories
between $\cZ(\Vect^{\ov{\omega}}_{H/N})$ and $\locmod_{\cZ(\Vect_G^\omega)}(A)$.
\end{Theorem}
\begin{proof}[Proof (sketch)]
Using Proposition~\ref{prop:LocmodA} and Remark~\ref{rem:simplify}, it suffices to show that the braided monoidal category $\locmod_{\cZ(\Vect_H^{\pi^*\ov{\omega}})}(B)$ is equivalent to $\cZ(\Vect_{H/N}^{\ov{\omega}})$ for $B=B(N,\kappa,\epsilon)$ an algebra as in Proposition~\ref{prop:B-algebra}. Article \cite{DS} produces a braided monoidal functor from the latter category to a~category of YD-compatible $H/N$-modules and comodules involving further cocycle data from $\wF^\bullet(H/N,H/N,\Bbbk^\times)$. It is then shown in \cite[Proposition A.1]{DS} that any such deformed monoidal category, when braided, is equivalent to $\cZ(\Vect_{H/N}^{\ov{\omega}})$. The proof does not rely on the assumption $\cha \Bbbk=0$.
\end{proof}

As by \cite{Sch}, $\cZ(\Rep_{\Vect_{G}^{\omega}}(A))\cong \locmod_{\cZ(\Vect_{G}^{\omega})}(A)$ one can ask if the equivalences of Theorem~\ref{thm:DSlocal} stems from an equivalence of the monoidal categories
$\Vect_{H/N}^{\ov{\omega}}$ and $\Rep_{\Vect_{G}^{\pi^*\ov{\omega}}}(A)$, see Definition~\ref{def:RepA-general}.
To the knowledge of the authors, this remains an open question in general, but see Proposition~\ref{prop:trivialdata} below for the case of trivial cocycle data $\epsilon\oplus\kappa $, and Section~\ref{sec:expl} for the case of odd dihedral groups.

\subsection{Special cases}

The following corollary expresses two extreme cases of Corollary~\ref{cor:modular}, when $N$ is as large or as small as possible. For this, we recall the \emph{Frobenius--Perron dimension} $\FPdim(\cC)$ of a finite tensor category $\cC$ \cite[Section~4.5]{EGNO} and objects within it, see Section~\ref{subsec:local}. It is well known that
\[\FPdim(\Vect_G)=\FPdim\big(\Vect_G^\omega\big)=|G| \qquad \text{and} \qquad \FPdim(\cZ\big(\Vect_G^\omega\big))=\FPdim\big(\Vect_G^\omega\big)^2=|G|^2.\]
It follows from \cite[Corollary 4.21]{LW3} that
\[\FPdim \big(\locmod_{\cZ(\Vect_G^\omega)}(A)\big)=\frac{|G|^2}{\FPdim_{\cZ(\Vect_G^\omega)}(A)^2}= \frac{|H|^2}{|N|^2},\]
using that $\cZ\big(\Vect_G^\omega\big)$ is non-degenerate \cite[Proposition~8.6.3]{EGNO}. Here, in order to compute $\FPdim_{\cZ(\Vect_G^\omega)}(A)$, we use the forgetful \emph{quasi}-tensor functor $\cZ\big(\Vect_G^\omega\big)\to \Vect$. In fact, \[\FPdim_{\cZ(\Vect_G^\omega)}(A)=\dim_\Bbbk (A)=\frac{|G||N|}{|H|},\]
using \cite[Proposition~4.5.7]{EGNO} in the first equality and the basis $a_{g_i,n}$ of $A$ in the second equality.
Moreover,
\[\FPdim \big(\Rep_{\cZ(\Vect_G^\omega)}(A)\big)=\frac{|G|^2}{\dim_\Bbbk(A)}=\frac{|G||H|}{|N|}.\]
Note that this shows that the categories
$\Rep_{\cZ(\Vect^\omega_H)}(B(N,\kappa,\epsilon))$ and $\Rep_{\cZ(\Vect^\omega_G)}(A(H,N,\kappa,\epsilon))$ are inequivalent if $G\neq H$.

\begin{Corollary} 
Let $A:=A(H,N,\kappa,\epsilon)$ be an algebra in $\cZ\big(\Vect_G^\omega\big)$ as defined in Definition~{\rm \ref{def:A}}.
\begin{enumerate}\itemsep=0pt
\item[$(a)$] Then $A$ is \emph{trivializing}, i.e., $\locmod_{\cZ(\Vect_{G}^\omega)}(A)\simeq \Vect$, if and only if $N=H$.
\item[$(b)$] If $N=\{1\}$, then $\locmod_{\cZ(\Vect_G^\omega)}(A)$ and $\cZ\big(\Vect_H^\omega\big)$ are equivalent ribbon categories.
\end{enumerate}
\end{Corollary}
\begin{proof}
With the above computations of FP dimensions, this follows as in \cite[Corollary~6.18]{LW3}, where part (b) uses the equivalence in Theorem~\ref{thm:Iequiv}.
\end{proof}

Next, we consider the special case when $\kappa$ and $\epsilon$ are both trivial.
\begin{Proposition}\label{prop:trivialdata}
 Let $N\triangleleft H\leq G$ be subgroups with $|N|\cdot |G:H|\in \Bbbk^\times$ and $\omega\in C^3(G,\Bbbk^\times)$ such that $\omega|_H=\pi^*\ov{\omega}$ for a $3$-cocycle $\ov{\omega}$ of $H/N$. Then the equivalence of tensor categories $T$ from Proposition~{\rm \ref{prop:RepAH-equiv}} induces an equivalence of tensor categories between $\Rep_{\Vect_G^\omega}(A(H,N,1,1))$ and~$\Vect_{H/N}^{\ov{\omega}}$.
\end{Proposition}
\begin{proof}
Denote $B=B(N,1,1)$ and $A=A(H,N,1,1)$. Then $T(B)\cong A$ as algebras in $\Vect_G^\omega$ via the algebra morphism that sends $\delta_{g_i}\otimes e_n$ to $a_{g_i,n}$. Moreover, both $T(B)$ and $A$ have the same $G$-grading. Thus, $T$ induces an equivalence of categories
\[T\colon\ \Rep_{\Vect_H^\omega}(B)\to \Rep_{\Vect_G^\omega}(A).\]
Explicitly, a right $B$-module $V$ in $\Vect_H^\omega$ is mapped to the right $A$-module, defined on the $G$-graded vector space $T(V)=A_H\otimes V$ with right $A$-action given by
\begin{align*} a_{T(V)}^r((\delta_{g_i}\otimes v)\otimes(\delta_{g_j}\otimes n))&=T(a^r_B)\mu_{V,B}^T((\delta_{g_i}\otimes v)\otimes(\delta_{g_j}\otimes n))\\
&=\frac{\delta_{i,j}}{\gamma(g_i)(|v|,n)}(\delta_{g_i}\otimes (v\cdot n)).
\end{align*}
We will equip this functor with a monoidal structure induced from $\mu_{V,W}^T$ in Lemma~\ref{lem:T-monoidal}. This way, $\mu_{V,W}^T$ is a morphism of right $A$-modules. Here, we regard $T(B)\cong A$ as a commutative algebra in $\cZ\big(\Vect_G^\omega\big)$ with the twisted YD module structure defined in Lemma~\ref{lem:AFrobenius}.
Explicitly, we compute the half-braiding $c^A$ of $A$ with $T(W)$, cf.\ Proposition~\ref{prop:twistedYD-ZVect}, as
\begin{align*}
 c^A_{T(W)}((\delta_{g_i}\otimes n)\otimes (\delta_{g_j}\otimes w)={}& (\delta_{g_j}\otimes w)\otimes \big(g_j|w|^{-1}g_j^{-1}\cdot (\delta_{g_i}\otimes n)\big)\\
 ={}&\tau\big(g_j|w|^{-1}g_j^{-1},g_i\big)(n)\tau(g_k,h)(n)^{-1}\\
 &\cdot \big((\delta_{g_j}\otimes w)\otimes \big(\delta_{g_k}\otimes hnh^{-1}\big)\big),
\end{align*}
where $g_{k}$ satisfies $g_j|w|^{-1}g_j^{-1}g_i=g_kh\in g_kH$, i.e., $h=g_k^{-1}g_j|w|^{-1}g_j^{-1}g_i\in H$. Thus, we find that the \emph{left} $T(B)$-action is given by
\begin{align*}
a^l_{T(W)}((\delta_{g_i}\otimes n)\otimes (\delta_{g_j}\otimes w))={}&a^r_{T(W)}c^{A}_{T(W)}((\delta_{g_i}\otimes n)\otimes (\delta_{g_j}\otimes w))\\
={}&T(a^r_B)\mu^T_{B,W} c^A_{T(W)}((\delta_{g_i}\otimes n)\otimes (\delta_{g_j}\otimes w))\\
={}&\frac{\tau\big(g_j|w|^{-1}g_j^{-1},g_i\big)(n)}{\tau(g_k,h)(n)}T(a_B^r)\mu^T_{N,W}\\
&\cdot \big((\delta_{g_j}\otimes w)\otimes \big(\delta_{g_k}\otimes hnh^{-1}\big)\big)\\
 ={}&\delta_{j,k}\frac{\tau\big(g_j|w|^{-1}g_j^{-1},g_i\big)(n)}{\tau(g_k,h)(n)^{-1}\gamma(g_j)\big(|w|,hnh^{-1}\big)}\\
 &\cdot \big(\delta_{g_j}\otimes \big(w\cdot hnh^{-1}\big)\big)\\
 ={}&\delta_{i,j}\frac{\tau\big(g_i|w|^{-1}g_i^{-1},g_i\big)(n)}{\tau\big(g_i,|w|^{-1}\big)(n)\gamma(g_i)\big(|w|,|w|^{-1}n|w|\big)}\\
 &\cdot \big(\delta_{g_i}\otimes \big(w\cdot|w|^{-1}n|w|\big)\big)\\
 ={}&\frac{\delta_{i,j}}{\gamma(g_i)(n,|w|)}\big(\delta_{g_i}\otimes \big(w\cdot|w|^{-1}n|w|\big)\big).
\end{align*}
Here, we used that if $k=j$ then $h=|w|^{-1}g_j^{-1}g_i\in H$ which implies that $g_i=g_j$ in the second equality. The last equality follows from equation~\eqref{eq:gamma-braid2}.

Now, we check that $\mu_{V,W}^T$ descents to a morphism
\[T(V)\otimes_A T(W)\to T(V\otimes_B W).\]
This follows by comparing
\begin{align*}
 &\mu^{T}_{V,W}\alpha_{T(V),T(B),T(N)}(\ide_{T(V)}\otimes a^l_{T(W)})(((\delta_{g_i}\otimes v)\otimes (\delta_{g_j}\otimes n))\otimes (\delta_{g_k}\otimes w))\\
 &\quad=\frac{\delta_{j,k}\omega\big(g_i|v|g_i^{-1},g_jng_j^{-1},g_k|w|g_k^{-1}\big)^{-1}}{\gamma(g_j)(n,|w|)}\mu^{T}_{V,W}\big((\delta_{g_i}\otimes v)\otimes_B \big(\delta_{g_j}\otimes \big(w\cdot|w|^{-1}n|w|\big)\big)\big)\\
 &\quad=\frac{\delta_{i,j}\delta_{j,k}\omega\big(g_i|v|g_i^{-1},g_ing_i^{-1},g_i|w|g_i^{-1}\big)^{-1}}{\gamma(g_i)(n,|w|)\gamma(g_i)(|v|,n|w|)}(\delta_{g_i}\otimes \big(v\otimes_B \big(w\cdot|w|^{-1}n|w|\big)\big)\\
 &\quad\stackrel{\text{Lemma~\ref{eq-gammacond}}}{=} \frac{\delta_{i,j}\delta_{j,k}\omega(|v|,n,|w|)^{-1}}{\gamma(g_i)(|v|,n)\gamma(g_i)(|v|n,|w|)}(\delta_{g_i}\otimes \big(v\otimes_B 0 \big(w\cdot |w|^{-1}n|w|\big)\big)\\
 &\quad=\frac{\delta_{i,j}\delta_{j,k}}{\gamma(g_i)(|v|,n)\gamma(g_i)(|v|n,|w|)}(\delta_{g_i}\otimes (v\cdot n\otimes_B w))\\
 &\quad=\mu_{V,W}^T(a_{T(V)}^r\otimes \ide_{T(W)})(((\delta_{g_i}\otimes v)\otimes (\delta_{g_j}\otimes n))\otimes (\delta_{g_k}\otimes w)),
\end{align*}
where the second-last equality uses the compatibility condition of the relative tensor product~${V\otimes_B W}$. Now, the induced morphism $\mu_{V,W}^T\colon T(V)\otimes_A T(W)\to T(V\otimes_B W)$ is directly checked to be an isomorphism. Coherence of $\mu^T$ with associators follows as in Proposition~\ref{prop:RepAH-equiv}. Thus, we have shown that $T$ gives an equivalence of tensor categories between $\Rep_{\Vect_H^\omega}(B)$ and~$\Rep_{\Vect_H^\omega}(A)$.

Finally, $B=R(\one)$ for the right adjoint functor $R$ used in Corollary~\ref{cor:gen-equiv}. Thus, $\Vect_{H/N}^{\ov{\omega}}$ is equivalent to $\Rep_{\Vect_H^\omega}(B)$ by \cite[Proposition 6.1]{BN11}. Composing these two tensor equivalences proves the claim.
\end{proof}

\subsection{Examples for odd dihedral groups}\label{sec:expl}
In this section, we provide a full list of isomorphism classes of rigid Frobenius algebras (or, connected \'etale algebras) in $\cZ\big(\Vect_G^\omega\big)$ in the case when $G$ is an odd dihedral group and $\omega$ any $3$-cocycle valued in $\Bbbk=U(1)\subseteq \mC$. Moreover, we determine the tensor categories of representations of these algebras.

Let $G = D_{2m+1}$, the dihedral group of odd degree $2m +1$ with presentation
\[\big\langle s,r \mid s^2 = r^{2m+1} = e ,\, sr = r^{-1}s \big\rangle.\]
\begin{Notation}
 Consider $g \in G=D_{2m+1}$. We will write $g$ in terms of the group generators, namely $g=s^{g_0}r^{g_1}$, where $g_0\in \lbrace 0,1 \rbrace$ and $g_1 \in \lbrace -m, -m+1, \dots , m-1, m \rbrace$.
\end{Notation}

We are going to use classify the rigid Frobenius algebras in $\cZ\big(\Vect_G^\omega\big)$.
By Theorem~\ref{thm:classification}, these rigid Frobenius algebras are of the form $A(H,N,\kappa,\epsilon)$ as defined in Definition~\ref{def:A}.

First, we determine the $3$-cocycle $\omega\colon G\times G \times G \to U(1)$.
By \cite[equation~3.2.8]{dWP}, there are $4m+2$ independent $3$-cocycles classes in $H^3(G,U(1))$, parametrized by $p \in \lbrace 0,1,\dots , 4m+1 \rbrace$. The explicit formula for the $3$-cocycle $\omega_p$ is, for $a,b,c \in G$, given by{\samepage
\begin{gather}
 \omega_p(a,b,c) :=\exp\Big(\tfrac{2\pi {\rm i} p}{(2m+1)^2}\big((-1)^{b_0+c_0}a_1((-1)^{c_0}b_1+c_1 - [(-1)^{c_0}b_1+c_1]) \nonumber\\
 \hphantom{\omega_p(a,b,c) :=\exp\Big(}{} + \tfrac{(2m+1)^2}{2}a_0b_0c_0\big) \Big).\label{eq:omega_p}
\end{gather}}

\noindent
Here, the rectangular bracket reduces the quantity modulo $2m+1$ in the range $\lbrace -m,\dots ,m \rbrace$.
We thus observe that $(-1)^{c_0}b_1+c_1 - [(-1)^{c_0}b_1+c_1] = l(2m+1)$ for $l\in \lbrace -1,0,1 \rbrace$. This allows us to simplify the above formula to
\begin{equation*}
 \omega_p(a,b,c) :=\exp\Big(\tfrac{2\pi {\rm i} p}{(2m+1)}\big((-1)^{b_0+c_0}a_1(l + \tfrac{(2m+1)}{2}a_0b_0c_0\big) \Big).
\end{equation*}
By Remark~\ref{rem:Ntriv}, we need to find values for $p$ such that $\omega_p$ is trivial when restricted to a normal subgroup $N\triangleleft H \subseteq G$. We now discuss the possible choices of $H$, $N$.

The subgroups of the odd degree dihedral group $D_{2m+1}$ are split into two types; either a~dihedral subgroup of odd degree $D_{(2m+1)/d}$, or a cyclic group of the form $\mZ_{(2m+1)/d} \cong \langle r^{d}\rangle$. Here, $d$ is a divisor of $2m+1$.
The normal subgroups of $D_{2m+1}$ are exactly the group itself, or the subgroups of cyclic form. Thus, we get three cases:
\begin{itemize}\itemsep=0pt
 \item $H= \mZ_{(2m+1)/d}$, $N= \mZ_{(2m+1)/(df)}$,
 \item $H=D_{(2m+1)/d}$, $N = \mZ_{(2m+1)/(df)}$,
 \item $H=N=D_{(2m+1)/d}$,
\end{itemize}
where $f$ is a divisor of $(2m+1)/d$. For ease of notation, we shall set
\[x:=(2m+1)/d \qquad \text{and}\qquad y:=(2m+1)/(df).\]
We shall now determine for which values of $p$, the cocycle $\omega_p$ will become trivial when restricted to $N$ in each case.

\begin{Lemma}
 In the cases such that $N= \mZ_{(2m+1)/(df)}$, $\omega_p|_N$ is trivial when
 \[p\equiv 0 \mod (2m+1)/(df).\]
\end{Lemma}
\begin{proof}
When we restrict to $N$, we can have that $g_0=0$, $g_1 = df g_2$ for all $g\in N$, where $g_2\in \lbrace -(y-1)/2, -(y+1)/2,\dots ,(y-3)/2 ,(y-1)/2 \rbrace$. Thus $\omega_p$ becomes \[\omega_p|_{\mZ_y}(a,b,c) = \exp\Big(2 \pi {\rm i}\tfrac{p l d f a_2}{(2m+1)} \Big).
\]
We require this restriction to be trivial for all values of $a_2$, $l$. This occurs only when $p\equiv 0 \mod y$. There are $2df$-choices of $p$ in the applicable range.
\end{proof}

\begin{Lemma}
In the case such that $N= D_{(2m+1)/d}$, $\omega_p|_N$ is trivial when $p\equiv 0 \mod (2m+1)/d$.
\end{Lemma}

\begin{proof} When considering $g\in N$, we observe that $g_1 = dg_2$, where $g_2\in \lbrace -(x-1)/2, -(x+1)/2,\dots ,(x-3)/2 ,(x-1)/2 \rbrace \rbrace$.
 Thus we get that
 \begin{equation*}
 \omega_p|_{D_x}(a,b,c) =\exp\Big(\tfrac{2\pi {\rm i} p d}{(2m+1)}\big((-1)^{b_0+c_0}a_2(l + \tfrac{(2m+1)}{2 d}a_0b_0c_0\big) \Big).
 \end{equation*}
 It can now be seen that this $3$-cocycle is trivial everywhere on $H$ only when $p\equiv 0 \mod x$. Thus there are $2d$-choices for $p$.
\end{proof}

From these lemmas, we are now in a position to classify all rigid Frobenius algebras in $\cZ(\Vect^{\omega_p}_{D_{2m+1}})$, where $p \equiv 0 \mod y$ for some divisors $d|(2m+1)$ and $f|x$, by finding all possible data for $\kappa$, $\epsilon$.
\begin{Lemma}
 In all cases, $\kappa$ is a trivial $2$-cocycle in $H^2(N,U(1))$.
\end{Lemma}
\begin{proof}
 By equation~\eqref{eq:tau}, $\kappa$ is a $2$-cocycle on $N$. In the case that $N=\mZ_{y}$, \cite[equation~2.3.14]{dWP} gives us that $H^2(\mZ_y,U(1)) \cong \{0\}$, so $\kappa$ is trivial up to coboundary.

In the case that $N=D_x$, we shall use the dual universal coefficient theorem \cite[Theorem~3.6.5]{Wei} to calculate the relevant cohomology group,
 \[H^2(D_x,U(1)) \cong \Hom(H_2(D_x,\mZ),U(1)) \oplus \Ext^1_{\mZ}(H_1(D_x,\mZ),U(1)).\]
By \cite[Example~6.8.5]{Wei}, the involved homology groups are $H_2(D_x,\mZ) \cong \lbrace 0 \rbrace$, $H_1(D_x,\mZ) \cong \mZ_2$. We get that \[H^2(D_x,U(1)) \cong \Ext^1_{\mZ}(\mZ_2,U(1)) \cong \lbrace 0\rbrace,\]
where the last isomorphism follows from \cite[Corollary 3.3.11]{Wei}. Thus $\kappa$ is again trivial up to coboundary.
\end{proof}

To begin determining $\epsilon$, we note that when $N=\mZ_y$ we can use equation~\eqref{eq:omega_p} to calculate that, for $h,g \in H$ and $a,b \in N$,
\begin{align*}
 &\omega_p(a,h,g) = 1 = \omega_p(h,a,b) ,\qquad
 \omega_p(h,g,a) = \omega_p\big(h,gag^{-1},g\big) ,
\end{align*}
and thus $\tau(h,g)(n) = 1 = \gamma(h)(n,m)$. We also get this result when $H=N=D_x$, as $\omega_p$ is trivial everywhere by construction.

Thus, the conditions $\epsilon$ must satisfy from Notation \ref{not:cocycledata} is now
\begin{gather}
 \epsilon_h\big(gng^{-1}\big)\epsilon_g(n)= \epsilon_{hg}(n)\label{excond1}\\
 \epsilon_h(n)\epsilon_h(m) = \epsilon_h(nm) \label{excond2}\\
 \epsilon_n(m) = 1 \label{excond3},
\end{gather}
as well as $\epsilon_h(1)=1$.

Equation~\eqref{excond2} states that, for any $h\in H$, $\epsilon_h$ is a $1$-cocycle valued in $C^1(N,U(1))$, where~$N$ acts trivially on $U(1)$. There are no non-trivial $1$-coboundaries in this construction and so~$C^1(N,U(1))=H^1(N,U(1))$.

We shall now determine the value of $\epsilon$ in all three cases.

\begin{Lemma}
 When $H=N = D_{(2m+1)/d}$, $\epsilon$ is the trivial function $\epsilon\colon H\times N \to U(1)$.
\end{Lemma}
\begin{proof}
 Follows immediately from equation~\eqref{excond3} as $H=N$.
\end{proof}

\begin{Lemma}
 When $H=\mZ_{(2m+1)/d},N=\mZ_{(2m+1)/(df)}$, $\epsilon$ is a $1$-cocycle in $H^1(H,N)\cong N$.
\end{Lemma}
\begin{proof}
 As $H$ is abelian, equation~\eqref{excond1} becomes
 \begin{equation}\label{excond4}
 \epsilon_h(n)\epsilon_g(n) = \epsilon_{hg}(n)
 \end{equation}
 and thus $\epsilon$ is a $1$-cocycle in $H^1\big(H,H^1(N,U(1))\big)$, where $H$ acts trivially.

 By \cite[equation~2.3.13]{dWP}, $H^1(\mZ_{y},U(1)) \cong \mZ_{y}$, and so $\epsilon \in H^1(H,N)$.
 We then use \cite[Theorem~3.6.5, Corollary 3.3.11, Example 6.2.3]{Wei} to calculate that
 \[H^1(\mZ_x,\mZ_y) \cong \Hom(\mZ_x,\mZ_y)\cong \mZ_y,\]
 where the last isomorphism follows as $y$ divides $x$.
\end{proof}

\begin{Lemma}
 When $H=D_{(2m+1)/d},N=\mZ_{(2m+1)/(df)}$, $\epsilon$ is a trivial $1$-cocycle in $H^1(H,N)$.
\end{Lemma}
\begin{proof}
We first note that we can construct a $\mZ_2$-grading on $H$ by forming the quotient group $H/\mZ_{2m+1}\cong \mZ_2$. When $h\in H$ is in the $0$-graded component (i.e, in $\mZ_{2m+1}$), it is clear that equation~\eqref{excond1} becomes equation~\eqref{excond4}.

When $h \in H$ is in the $1$-graded component, equation~\eqref{excond1} becomes \[\epsilon_h(n^{-1})\epsilon_g(n) = \epsilon_{hg}(n).\]
By setting $g=1$, we observe that $\epsilon_h(n)=\epsilon_h(n^{-1})$. Thus, equation~\eqref{excond1} becomes equation~\eqref{excond4} once more. Thus, $\epsilon \in H^1\big(H,H^1(N,U(1))\big)$, where $H$ acts trivially, and as in the previous lemma, $H^1(\mZ_y,U(1)) \cong \mZ_y$. We then use \cite[Theorem~3.6.5, Corollary~3.3.11, Example~6.8.5]{Wei} to calculate that
\[H^1(D_x,\mZ_y)\cong \Hom(\mZ_2,\mZ_y) \cong \lbrace 0 \rbrace,\]
with the last isomorphism following as $y$ is odd and so $N=\mZ_y$ contains no non-identity elements of order $2$. Thus $\epsilon$ is trivial in $H^1(H,\mZ_y)$.
\end{proof}

We have thus found all rigid Frobenius algebras in $\cZ\big(\Vect_{D_{2m+1}}^{\omega_p}\big)$, up to isomorphism of algebras, proving the following proposition.

\begin{Proposition}
 Let $G=D_{2m+1}$, the dihedral group of odd degree $2m+1$, and let $d$, $f$ be a~pair of not necessarily proper divisors of $2m+1$ and $(2m+1)/d$, respectively.
 \begin{enumerate}\itemsep=0pt
 \item[$(a)$]
 Then, whenever $p\equiv 0 \mod (2m+1)/(df)$, there exist rigid Frobenius algebras $\cZ\big(\Vect_{D_{2m+1}}^{\omega_p}\big)$ of the form
 \[A\big(D_{(2m+1)/d},\mZ_{(2m+1)/(df)},1,1\big)\qquad \text{and}\qquad
A\big(\mZ_{(2m+1)/d},\mZ_{(2m+1)/(df)},\epsilon,1\big),\]
 where $\epsilon$ is a $1$-cocycle in $ H^1\big(\mZ_{(2m+1)/d}, \mZ_{(2m+1)/(df)}\big)\cong \mZ_{(2m+1)/(df)}$.
 \item[$(b)$] Additionally, there is a trivializing rigid Frobenius algebra of the form
 \[A\big(D_{(2m+1)/d},D_{(2m+1)/d},1,1\big)\]
 in $\cZ\big(\Vect_{D_{(2m+1)}}^{\omega_p}\big)$ whenever $p\equiv 0 \mod (2m+1)/d$.
 \end{enumerate}
This completely classifies all rigid Frobenius algebras in categories of the form $\cZ\big(\Vect_G^{\omega_p}\big)$, up to an isomorphism of algebras in $\cZ\big(\Vect_G^{\omega_p}\big)$.
\end{Proposition}

These algebras have the structure of
 a $\mC$-vector space with $\mC$-basis ${\lbrace a_{g,n} \mid g \in G, \allowbreak n \in N\rbrace}$ subject to the relations
 \[a_{gh,n} = \epsilon_h(n) a_{g,hnh^{-1}} \qquad \forall h\in H\]
 and with
the following YD module and algebra structures in their respective categories $\cZ\big(\Vect^{\omega_p}_H\big)$:
 \begin{enumerate}\itemsep=0pt
 \item[$(i)$] $G$-action: $k \cdot a_{g,n}=a_{kg,n}$ for $k\in G$;
 \item[$(ii)$] $G$-coaction: $\delta(a_{g,n})=gng^{-1} \otimes a_{g,n}$;
 \item[$(iii)$] multiplication: $a_{g,n} a_{g,m} = a_{g,nm}$ for $g \in G$ and $ n,m\in N$, and $a_{g,n}a_{k,n} =0$ if $kH\neq gH$;
 \item[$(iv)$] unit: $1_A =\sum_{i\in I}a_{g_i,1}$;
 \item[$(v)$] coproduct: $\Delta_A (a_{g,n}) = \sum_{m\in N}a_{g,m}\otimes a_{g,m^{-1}n}$ for all $g\in G$ and $n\in N$;
 \item[$(vi)$] counit:
 $\varepsilon_A(a_{g,n}) = \delta_{n,1}.$
 \end{enumerate}
 We note that all of these algebras are images of group algebras $\Bbbk N$ under the functor $I$, with the subgroup $H$ governing the resulting algebra multiplication in $\cZ\big(\Vect_G^{\omega_p}\big)$.

Furthermore, for all of the cases where $\epsilon$ is trivial, we can utilise Proposition~\ref{prop:trivialdata} to determine their categories of representations when viewed as objects in $\Vect_G^{\omega_p}$, up to tensor equivalence. Explicitly,
\begin{itemize}\itemsep=0pt
 \item for $A=A(D_{(2m+1)/d},\mZ_{(2m+1)/(df)},1,1)$, $\Rep_{\Vect_G^{\omega_p}}(A) \cong \Vect^{\ov{\omega}}_{D_f}$;
 \item for $A=A(\mZ_{(2m+1)/d},\mZ_{(2m+1)/(df)},\epsilon,1)$, $\Rep_{\Vect_G^{\omega_p}}(A) \cong \Vect^{\ov{\omega}}_{\mZ_f}$;
 \item for $A=A(D_{(2m+1)/d},D_{(2m+1)/d},1,1)$, $\Rep_{\Vect_G^{\omega_p}}(A) \cong \Vect^{\ov{\omega}}_{\lbrace 0\rbrace} = \Vect_\mC$.
\end{itemize}
 Even if $\epsilon$ is non-trivial in the second case, since $\kappa$ is trivial, the categories
 $\Rep_{\Vect_G^{\omega_p}}(A)$ do not depend on $\epsilon$, only their associativity isomorphisms does, which, in any case, corresponds to a~$3$-cocycle on $\mZ_f$. By Remark~\ref{rem:simplify}, we obtain all possible $3$-cocycles of $D_f$ and $\mZ_f$ as $\ov{\omega}$, up to coboundary.

 Note that by \cite{Sch}, the corresponding categories of local modules in $\cZ\big(\Vect_G^\omega\big)$ are equivalent to the Drinfeld centers $\cZ\big(\Vect_{D_f}^{\ov{\omega}}\big)$, respectively, $\cZ\big(\Vect_{\mZ_f}^{\ov{\omega}}\big)$ in the first two cases (see Theorem~\ref{thm:DSlocal}), and we recover the fact that the case $H=N=D_{(2m+1)/d}$ gives a trivializing algebra in the third case.

\appendix
\section{Group cohomology}\label{appendix}

\subsection{Definitions}

Here, we collect basic definitions from group cohomology used in the text, see, e.g., \cite[Section~3.4]{Ben}, \cite[Chapter~III]{Bro}.\footnote{These references typically use left module conventions.} Let $G$ be a group, the \emph{bar resolution} is the complex
\[\cdots \mZ G^n \otimes_{\mZ} \mZ G\xrightarrow{\partial_n} \mZ G^{n-1}\otimes_{\mZ}\mZ G \xrightarrow{\partial_{n-1}} \cdots \mZ G^1\otimes_{\mZ}\mZ G \xrightarrow{\partial_1}\mZ G,\]
where $ \mZ G^n\otimes_{\mZ} \mZ G$ is a right $\mZ G$-module via right multiplication. As an $\mZ G$-module, $\mZ G^n\otimes_{\mZ} \mZ G $ is freely generated by $n$-tuples $(g_1,\dots, g_n)$. The differential is the $\mZ G$-module homomorphism determined by
\[\partial^n(g_1,\dots, g_n)=(-1)^n(g_2,\dots, g_n)+ \sum_{i=1}^{n-1}(-1)^{n-i} (g_1,\dots, g_ig_{i+1},\dots, g_n)+ (g_1, \dots, g_{n-1})g_n.\]
Given a right $\mZ G$-module $A$, we obtain the cochain complex $F^\bullet(G,A)$ on abelian groups of functions $F^n(G,A)=\Fun(G^n,A)$ with differentials
\[ M \xrightarrow{\deri^0} F^1( G,A) \xrightarrow{\deri^1} F^2(G,A) \cdots F^{n-1}(G,A)\xrightarrow{\deri^{n-1}}
F^n(G,A)\cdots, \]
where $\deri^n$ is obtained by composing with $\partial_{n+1}$ under the identification \[\Hom_{\mZ G}(\mZ G^n\otimes_{\mZ} \mZ G,A)\cong \Fun(G^n,A)=F^n(G,A),\]
where the latter is simply the $\mZ$-module of maps $G^n\to A$. Explicitly, the differential $\deri=\deri^n$ is given on a map $\omega\colon G^n\to A$ by
\begin{equation*}
 \begin{split}\deri \omega(g_0,\dots, g_{n})={}& (-1)^{n+1}\omega(g_1,\dots, g_{n})\\
 &{}+\sum_{i=0}^{n-1}(-1)^{n-i}\omega(g_0,\dots, g_ig_{i+1}, \dots, g_{n})+ \omega(g_0,\dots, g_{n-1})\cdot g_{n},
 \end{split}
\end{equation*}
where $\cdot$ denotes the action of $G$ on $A$. In practice, we often use the $G$-module $A=\Bbbk^\times$ (or $U(1)$), with trivial $G$-action. In this case, we use multiplicative notation.
We denote $C^n(G,A):=\ker(\deri^n)$ for the space of \emph{$n$-cocycles} and the \emph{$n$-th cohomology group} is
\[H^n(G,A):=C^n(G,A)/\Img \deri^{n-1}.\]
For example, a $3$-cocycle with values in $\Bbbk^\times$ satisfies equation~\eqref{3cocycle}.

\subsection{3-cocycle identities}

Let $\omega\colon G^3\to \Bbbk^\times$ be a $3$-cocycle in group cohomology (computed using the bar resolution).
The $3$-cocycle condition on $\omega$ is
\begin{align}
 \omega ( g_1 g_2,g_3,g_4 ) \omega( g_1,g_2,g_3 g_4 )=\omega ( g_1,g_2,g_3 ) \omega( g_1,g_2 g_3,g_4 ) \omega ( g_2,g_3,g_4 ).
 \label{3cocycle}
\end{align}
We assume that $\omega$ is \emph{normalized}, i.e., $\omega(g,h,k)=1$ as soon as one of the entries is the identity of~$G$.
In what follows we provide proofs for several identities we have used along the way involving cocycles, $\tau$ (as defined in equation \eqref{eq-taudef}) and $\gamma$ (as defined in equation \eqref{eq-gammadef}).

\begin{Lemma}\label{eq-taucond}The map
$\tau (h,k)(d)$ satisfies
\begin{equation*}
 \tau(h,k)(d) \tau(g,hk)(d)=\tau(gh,k)(d) \tau(g,h)\big(kdk^{-1}\big), \qquad \forall g,h,k,d\in G.
\end{equation*}
\end{Lemma}

\begin{proof}
This equation follows from repeatedly applying the $3$-cocycle in equation \eqref{3cocycle} with the following entries:
\begin{itemize}\itemsep=0pt
 \item [--] $g_1=g$, $g_2=h$, $g_3=k$, $g_4=d$,
 \item [--] $g_1=g$, $g_2=h$, $g_3=kdk^{-1}$, $g_4=k$,
 \item [--] $g_1=g$, $g_2=hkdk^{-1}h^{-1}$, $g_3=h$, $g_4=k$,
 \item [--] $g_1=ghkdk^{-1}h^{-1}g^{-1}$, $g_2=g$, $g_3=h$, $g_4=k$.\hfill $\qed$
\end{itemize}\renewcommand{\qed}{}
\end{proof}

\begin{Lemma}\label{eq-gammacond} The map $\gamma (h)(g,g')$ is related to the $3$-cocycle $\omega(g,g',g'')$ via the following identity:
\begin{equation*}
\frac{\gamma(h)(gg',g'') \gamma(h)(g,g')}{\omega\big(hgh^{-1},hg'h^{-1},hg''h^{-1}\big) }=\frac{\gamma(h)(g,g'g'') \gamma(h)(g',g'')}{ \omega(g,g',g'')}.
\end{equation*}
\end{Lemma}

\begin{proof}
This equation follows from applying the $3$-cocycle in equation \eqref{3cocycle} several times with the following entries:
\begin{itemize}\itemsep=0pt
 \item[--] $g_1=h$, $g_2=g$, $g_3=g'$, $g_4=g''$,
 \item[--] $g_1=hgh^{-1}$, $g_2=hg'h^{-1}$, $g_3=hg''h^{-1}$, $g_4=h$,
 \item[--] $g_1=hgh^{-1}$, $g_2=h$, $g_3=g'$, $g_4=g''$, and
 \item[--] $g_1=hgh^{-1}$, $g_2=hg'h^{-1}$, $g_3=h$, $g_4=g''$.\hfill $\qed$
\end{itemize}\renewcommand{\qed}{}
\end{proof}

\begin{Lemma}\label{gammatau} The maps $\tau (h,k)(d)$ and $\gamma(k)(d,g)$ are related via the following identity:
 \begin{align*}
 \gamma ( k ) ( d,g ) \gamma ( h ) \big( kdk^{-1},kgk^{-1} \big) \tau ( h,k ) ( d ) \tau ( h,k ) ( g ) =\tau ( h,k ) ( dg ) \gamma ( hk ) ( d,g ).
 \end{align*}
\end{Lemma}
 \begin{proof}
 Proving this equality amounts to apply the $3$-cocycle in equation \eqref{3cocycle} with the following set of entries:
 \begin{itemize}\itemsep=0pt
 \item[--] $g_1=h$, $g_2=k$, $g_3=d$, $g_4=g$,
 \item[--] $g_1=h$, $g_2=kdk^{-1}$, $g_3=kgk^{-1}$, $g_4=k$,
 \item[--] $g_1=hkd(hk)^{-1}$, $g_2=hkg(hk)^{-1}$, $g_3=h$, $g_4=k$,
 \item[--] $g_1=h$, $g_2=kdk^{-1}$, $g_3=k$, $g_4=g$,
 \item[--] $g_1=hkd(kh)^{-1}$, $g_2=h$, $g_3=k$, $g_4=g$, and
 \item[--] $g_1=hkd(hk)^{-1}$, $g_2=h$, $g_3=kgk^{-1}$, $g_4=k$.\hfill $\qed$
\end{itemize}\renewcommand{\qed}{}
 \end{proof}

\subsection{Cohomology of crossed products of groups}\label{appendix:coh-crossed}

Let $H$, $G$ be groups together with a left action of $H$ on $G$ by group automorphisms, i.e., $H\mapsto \Aut(G)$, $h\mapsto \big(g\mapsto {}^hg\big)$. Then we can form the \emph{crossed product}
$G\rtimes H$, which is $G\times H$ as a~set with multiplication given by
\[(g_1,h_1)\cdot (g_2,h_2)=\big(g_1{}^{h_1}g_2,h_1h_2\big).\]
Let $A$ be a right $\mZ G$-module. Then $F^n(G,A)=\Fun(G^n,A)$ becomes a right $H$-module with action
\[(f\cdot h)(g_1,\dots, g_n)=f\big({}^hg_1,\dots, {}^hg_n\big).\]
Following \cite[Appendix~A]{DS}, define a double complex
\[F^{n,m}(H,G,A)=\Fun(H^{n},\Fun(G^m,A))=F^n(H,F^m(G,A)).\]
The two differentials are denoted by
\[\deri^{n,m}\colon\ F^{n,m}(H,G,A)\to F^{n+1,m}(H,G,A), \qquad \partial^{n,m} \colon\ F^{n,m}(H,G,A)\to F^{n,m+1}(H,G,A),\]
where
\[(\partial^{n,m}(f))(h_1,\dots, h_n)=\deri^n(f(h_1,\dots, h_n)).\]
The differentials commute, i.e., $\deri^{n,m+1}\partial^{n,m}=\partial^{n+1,m}\deri^{n,m}$ making $F^{\bullet,\bullet}(H,G,A)$ a double complex. Hence, one can consider the associated \emph{truncated} double complex
\begin{gather*}
\wF_{\Tot}^{n}(H,G,A)=\bigoplus_{i=0}^{n-1}F^{n-i,i}(H,G,A),\\
\deri_{\Tot}^n(f):=\deri^{n-i,i}(f)+(-1)^i\partial^{n-i,i}(f) \qquad \text{for $f\in F^{n-i,i}(H,G,A)$ with $i<n$}.
\end{gather*}
We will typically denote an element $f\in \wF^{i,n-i}(H,G,A)\subseteq F_{\Tot}^{n}(H,G,A)$ by a function
$f\colon H^i\times G^{n-i} \to A$.

Letting $G\rtimes H$ act on $A$ via the surjective homomorphism $G\rtimes H\to G$, the \emph{untrucated} total complex
\[F_{\Tot}^n(H,G,A)=\bigoplus_{i=0}^{n} F^{n-i,i}(H,G,A)\] is quasi-isomorphic to the complex $F^\bullet(G\rtimes H,A)$ computing group cohomology, see \cite{HS}.

Several cocycles considered in this paper have interpretations as elements of the truncated total complex $\wF_{\Tot}^{\bullet}$ with $A=\Bbbk^\times$.

\begin{Example}
We now let a subgroup $G$ act on itself via conjugation, while $G$ acts on $\Bbbk^\times$ trivially, using multiplicative notation. Consider a triple
\[T(\omega)=\tau\oplus \gamma \oplus \omega\in F^{2,1} \oplus F^{1,2}\oplus F^{0,3}=\wF^3_{\Tot}(G,G,\Bbbk^\times),\]
with $\gamma(h_1,g_1,g_2)=\gamma(h_1)(g_1,g_2)$ and $\tau(h_1,h_2,g_1)=\tau(h_1,h_2)(g_1)$ defined in equations~\eqref{eq-taudef} and~\eqref{eq-gammadef}.
Then
$T(\omega)$ is a $3$-cocycle in the totalized complex $C^3_{\Tot}(G,G,\Bbbk^\times)$ if and only if the following conditions hold
\begin{alignat*}{4}
&\deri^{2,1}(\tau)=1 && \Longleftrightarrow && \text{Lemma~\ref{eq-taucond}},&\\
&\partial^{2,1}(\tau)\deri^{1,2}(\gamma)=1 && \Longleftrightarrow && \text{Lemma~\ref{gammatau}},&\\
&\partial^{1,2}(\gamma)^{-1}\deri^{0,3}(\omega)=1\quad && \Longleftrightarrow\quad && \text{Lemma~\ref{eq-gammacond}},&\\
&\partial^{0,3}(\omega)=1 && \Longleftrightarrow && \text{equation~\eqref{3cocycle}}.&
\end{alignat*}
\end{Example}

\begin{Example}
Let $N\triangleleft H$ be a normal subgroup and let $H$ act on $N$ by conjugation. We consider $2$-boundaries in the complex $\wF_{\Tot}^{\bullet}(H,N,\Bbbk^\times)$. These can be parametrized by pairs
\[\epsilon \oplus \kappa\in F^{1,1}\oplus F^{0,2}=F^2_{\Tot}(H,N,\Bbbk^\times).\]
The total differential has three components, namely
\[\deri_{\Tot}^2(\epsilon \oplus \kappa)=\deri^{1,1}(\epsilon)\oplus \frac{\deri^{0,2}(\kappa)}{\partial^{1,1}(\epsilon)}\oplus\partial^{0,2}(\kappa).\]
Explicit formulas for the components are derived from
\begin{gather*}
\deri^{1,1}\epsilon(h_1,h_2,n_1)=\frac{ \epsilon\big(h_1,h_2n_1h_2^{-1}\big)\epsilon(h_2,n_1)}{\epsilon(h_1h_2,n_1)},\\
\partial^{1,1}\epsilon(h_1,n_1,n_2)=\frac{\epsilon(h_1,n_1)\epsilon(h_1,n_2)}{\epsilon(h_1,n_1n_2)}, \qquad
\deri^{0,2}\kappa(h_1,n_1,n_2)=\frac{\kappa\big(h_1n_1h_1^{-1},h_1n_2h_1^{-1}\big)}{\kappa(n_1,n_2)},\\
\partial^{0,2}\kappa(n_1,n_2,n_3)=\frac{\kappa(n_1,n_1)\kappa(n_1n_2,n_3)}{\kappa(n_1,n_2n_3)\kappa(n_2,n_3)}.
\end{gather*}
\end{Example}

\subsection*{Acknowledgements}

S.H.\ is supported by Engineering and Physical Sciences Research Council.
R.L.\ was supported by a Nottingham Research Fellowship.
 A.R.C.\ is supported by Cardiff University. The authors would like to especially thank the anonymous referees for their helpful comments and suggestions.

\pdfbookmark[1]{References}{ref}
\LastPageEnding

\end{document}